\documentclass{article}
\usepackage[utf8]{inputenc}
\usepackage{amsthm}
\usepackage{amsmath}
\usepackage{amssymb}
\usepackage{enumitem}
\usepackage{bookmark}
\RequirePackage{geometry}
\usepackage{caption}
\usepackage{graphicx}
\usepackage{hyperref}
\usepackage[T1]{fontenc}
\usepackage{amsfonts}
\usepackage{braket}
\usepackage[l2tabu, orthodox]{nag}
\usepackage{pdf14}
\usepackage{xpatch}
\RequirePackage{tabto}
\usepackage{indentfirst}
\usepackage{amscd,stmaryrd,latexsym}
\usepackage{xcolor}

\newtheorem{thm}{Theorem}
\newtheorem{lem}{Lemma}[section]

\newtheorem{Prop}{Proposition}[section]
\newtheorem{Prop*}{Proposition}

\newtheorem*{theorem*}{Theorem}
\newtheorem*{corollary*}{Corollary}
\newtheorem*{oss*}{Remark}
\theoremstyle{definition}

\newtheorem{oss}{Remark}[section]

\newcommand{\R}{\mathbb{R}}
\newcommand{\N}{\mathbb{N}}

\newcommand\res{\mathop{\hbox{\vrule height 7pt width .5pt depth 0pt
\vrule height .5pt width 6pt depth 0pt}}\nolimits}

\DeclareMathAlphabet{\mathscr}{OT1}{pzc}{m}{it}

\begin{document} 
 
\title{\textbf{Smooth approximations for constant-mean-curvature hypersurfaces with isolated singularities}} 
\author{Costante Bellettini\footnote{Department of Mathematics, University College London, United Kingdom.} \quad  \quad Konstantinos Leskas\footnote{Department of Mathematics, National and Kapodistrian University of Athens, Greece.}}
\date{}

\maketitle

\begin{abstract}
We consider a CMC hypersurface with an isolated singular point at which the tangent cone is regular, and such that, in a neighbourhood of said point, the hypersurface is the boundary of a Caccioppoli set that minimises the standard prescribed-mean-curvature functional. We prove that in a ball centred at the singularity there exists a sequence of smooth CMC hypersurfaces, with the same prescribed mean curvature, that converge to the given one. Moreover, these hypersurfaces arise as boundaries of minimisers. In ambient dimension $8$ the condition on the cone is redundant. (When the mean curvature vanishes identically, the result is the well-known Hardt--Simon approximation theorem.) 
\end{abstract}

\section{Introduction}

It is well known that variational constructions for area-type functionals may lead to singularity formation. Already in the widely studied case of area minimisation for hypersurfaces, if the ambient dimension is $8$ or higher, solutions cannot be expected to be completely smooth. The case of volume-constrained perimeter minimisation, which leads to isoperimetric regions, is analogous: in $\R^{n+1}$, or more generally in an $(n+1)$-dimensional Riemannian manifold, such regions have boundaries that are smoothly embedded away from a possible singular set of dimension at most $(n-7)$; when $n=7$ the singular set is made more precisely of isolated points. The phenomenon arises yet again in the case of minimax constructions for prescribed-mean-curvature functionals.

Examples show that this singular set is in general unavoidable. The well-known minimal cone $C_{4,4}=\{(x,y) \in \R^4 \times \R^4 \equiv \R^8: |x|^2=|y|^2\}$ (shown to be stable by Simons \cite{Simons}), is smooth away from the isolated singularity at the origin, and is area-minimising, e.g.~in any ball $B \subset \R^8$, with respect to the boundary condition $C_{4,4} \cap \partial B$. This was proved by Bombieri--De Giorgi--Giusti (\cite{BDG}, see also a more straightforward proof in \cite{DP}). This cone is in fact the unique minimiser for said boundary condition.
An isoperimetric region with two isolated singular points in an $8$-dimensional Riemannian manifold was recently constructed in \cite{Niu}.

On the other hand, it is fruitful to ask whether the appearance of singularities is a generic phenomenon. This question led to very important progress already in the 80s and has received renewed attention in recent years. The fundamental work by Hardt--Simon \cite{HS} shows an instance of generic regularity for solutions to the Plateau problem, in the following sense. Let a $7$-dimensional area minimiser in $\R^{8}$ be given, with (prescribed) $6$-dimensional smooth boundary $\Gamma$, and with an isolated singular point; then a slight perturbation of $\Gamma$ yields a minimiser that is completely smooth. This type of result lends itself to geometric applications, by shifting the genericity condition onto the Riemannian metric, as exemplified by Smale's proof of generic regularity of area-minimisers in any non-zero homology class \cite{Smale}. (We also refer the reader to \cite{CLS, LW}.) Very recently, the question of generic regularity for area minimisers has found affirmative answer in ambient dimension $9$ and $10$, in the work by Chodosh--Mantoulidis--Schulze \cite{CMS}, making progress on a long-standing conjecture (\cite{Smale})\footnote{After the appearance of this article, dimension $11$ was also addressed, see \cite{CMSW}.}. We also refer the reader to \cite{CCMS} and references therein, for generic regularity in the setting of mean curvature flow.

\medskip

Our main goal here is to prove a (local) smooth approximation result in the constant-mean-curvature (CMC) case, establishing a generic regularity result for the CMC Plateau problem analogous to the one proven in \cite{HS} (in particular, if the mean curvature vanishes identically, the Hardt--Simon theorem gives the result).
The variational setting for CMC hypersurfaces involves an energy that we will denote by $J_\lambda$, where $\lambda \in \R$ is the prescribed constant value of the scalar mean curvature. Roughly speaking, $J_\lambda$ evaluates the $n$-dimensional area of the hypersurface, from which it subtracts $\lambda$ times the $(n+1)$-volume enclosed by it. A natural way to formalise this is by working with boundaries of sets with finite perimeter. We briefly recall the relevant notions (with more details in Section \ref{variational_setup} below).

Let $E \subset U$ be a set with locally finite perimeter in a bounded open set $U \subset \R^{n+1}$, and let $\lambda \in \R$. We denote by $J_\lambda$ the functional (defined on any set $D \subset U$ with locally finite perimeter in $U$),
\[J_\lambda (D) = \text{Per}_{U}(D) -\lambda |D|,\]
where the notation $|D|$ stands for $\mathcal{L}^{n+1}(D)$. Given $W \subset \subset U$, the set $E$ is said to be a minimiser of $J_\lambda$ in $W  \subset U$ if it attains the following infimum:
\[\inf \{J_\lambda(D): D  \cap (U \setminus W) = E  \cap (U \setminus W)   \}.\]
In other words, the class of competitors for $E$ is that of sets (with locally finite perimeter in $U$) that coincide with $E$ outside $W$. Equalities and inclusions between sets of locally finite perimeter are always understood to hold in the $\mathcal{L}^{n+1}$-a.e.~sense. Prescribing the set in $U \setminus W$ amounts to fixing the boundary condition for the Plateau problem in $W$ (as customary in the setting of Caccioppoli sets).

If $E$ is a minimiser of $J_\lambda$ in $W \subset \subset U$, it is well-known (see e.g.~\cite{Mag, GMT, BeWi1}) that 
there exists a set $\Sigma \subset W$ with $\text{dim}_{\mathcal{H}}(\Sigma) \leq n-7$, such that $(\overline {\partial^* E} \cap W^\prime) \setminus \Sigma$ is smoothly embedded in $W'$ for every open set $W' \subset \subset W$, and that $(\overline {\partial^* E} \cap W^\prime) \setminus \Sigma$ has constant scalar mean curvature equal to $\lambda$. Here $\partial^* E$ denotes the reduced boundary of the set $E$. (More precisely, the mean curvature vector is $\lambda \nu_E$, where $\nu_E$ is the unit normal pointing into $E$.)

\medskip

The most immediate instance of our result states the following. 

\begin{thm}
\label{thm:main_dim8}
Let $E$ be a set with locally finite perimeter in an open set $U\subset \R^8$, and assume that $E$ minimises $J_\lambda$ in a ball $\hat{B} \subset \subset U$, for a given $\lambda \in \R$. There exists a ball $B \subset \hat{B}$, with the same centre, and a sequence of hypersurfaces $T_j$ smoothly embedded in $B$, with scalar mean curvature $\lambda$, and with $T_j \to \partial^* E$ in $B$. (The convergence holds in the sense of currents, in the sense of varifolds, as well as in the Hausdorff distance sense.) Moreover, $T_j = \partial^* E_j$, where each $E_j$ is a set with finite perimeter in $B$ and $\partial^* E_j$ stands for the reduced boundary of $E_j$ in $B$, and we have $E_j \subset E$ and $E_j \to E$ in $B$.
\end{thm}

We remark that the significance of Theorem \ref{thm:main_dim8} lies in the fact that the centre $p$ of $\hat{B}$ may be a singular point of $\overline{\partial^* E}$.

In ambient dimension $8$, as in Theorem \ref{thm:main_dim8}, isolated singular points are the only type of interior singularities that $\overline{\partial^* E}$ may possess. This is no longer the case when the ambient dimension is higher. Just as in \cite{HS}, we can remove the dimensional restriction in Theorem \ref{thm:main_dim8} by (strongly) restricting the singular behaviour of $E$ (Theorem \ref{thm:main_gendim} below). We work in a neighbourhood of an isolated (interior) singular point $p$ of $\overline{\partial^* E}$, with the further property that the multiplicity-$1$ varifold associated to $\partial^* E$, denoted by $|\partial^* E|$, admits a tangent cone at $p$ that is regular. We recall that a cone is regular when it is smooth away from the vertex, and the multiplicity is $1$ on the smooth part.

\begin{thm}
\label{thm:main_gendim}
Let $E$ be a set with locally finite perimeter in an open set $U\subset \R^{n+1}$, with $n\geq 7$, and assume that $E$ minimises $J_\lambda$ in a ball $\hat{B} \subset \subset U$, for a given $\lambda \in \R$. Assume furthermore that the centre $p$ of $\hat{B}$ is an isolated singularity of $|\partial^* E|$ and that $|\partial^* E|$ admits a tangent cone at $p$ that is regular (in the sense of varifolds). 

There exists a ball $B \subset \hat{B}$, with the same centre $p$, and a sequence of hypersurfaces $T_j$ smoothly embedded in $B$, with scalar mean curvature $\lambda$, and with $T_j \to \partial^* E$ in $B$. (The convergence holds in the sense of currents, in the sense of varifolds, as well as in the Hausdorff distance sense.) Moreover, $T_j = \partial^* E_j$, where each $E_j$ is a set with finite perimeter in $B$ and $\partial^* E_j$ stands for the reduced boundary of $E_j$ in $B$, and we have $E_j \subset E$ and $E_j \to E$ in $B$.
\end{thm}

\begin{oss}
By construction, for each $j$ the set $E_j$ is a minimiser, more precisely, it is given by $\hat{E}_j \cap B$ for a set with finite perimeter $\hat{E}_j \subset \hat{B}$ that minimises $J_\lambda$ in $B \subset \hat{B}$ (among sets that coincide with $\hat{E}_j$ in $\hat{B}\setminus B$). The mean curvature vector of $|\partial^* E_j|$ in $B$ is given by $\lambda \nu_{E_j}$, where $\nu_{E_j}$ is the inward pointing unit normal.
\end{oss}

\begin{oss}
The regularity theory for $n=7$ implies not only that the singular set is made of isolated points, but also that any varifold tangent cone (at a singular point) must be regular, via a standard dimension reduction argument. Therefore Theorem \ref{thm:main_dim8} follows from Theorem \ref{thm:main_gendim}.
\end{oss}

\begin{oss}
In the special case $\lambda=0$ Theorems \ref{thm:main_dim8} and \ref{thm:main_gendim} were proved in \cite{HS} (see also \cite{CMS}). Our proof relies on the result for $\lambda=0$.
\end{oss}

\begin{oss}
In both Theorems \ref{thm:main_dim8} and \ref{thm:main_gendim}, the convergence $T_j \to \partial^* E$ is strong (graphical and $C^2$) in $B \setminus \{p\}$, thanks to Allard's regularity theorem and standard elliptic PDE theory.
\end{oss}

\begin{oss}
Theorems \ref{thm:main_dim8} and \ref{thm:main_gendim} lend themselves applications in geometry, such as the surgery procedure in \cite{BMS} (where a generic existence result for smooth CMC closed hypersurfaces in compact Riemannian $8$-dimensional manifolds is proved).
\end{oss}

\medskip

In proving Theorem \ref{thm:main_gendim} (which we will do in Section \ref{approx}, see Theorem \ref{app thm}) we establish a result of independent interest on the existence and regularity of minimisers of $J_\lambda$, for the CMC Plateau problem. We present here a simplified version (sufficient for its scope within the proof of Theorem \ref{thm:main_gendim}). The more general result requires some notation and will be given in Theorem \ref{thm:variational} of Section \ref{variational_setup}.

\begin{thm}
\label{thm:variational_simplified}
Let $E_0$ be a set with finite perimeter in $U=B^{n+1}_{R}(p)$. Let $\lambda \in (0, \infty)$ and $r \in (0, \frac{n}{\lambda})$, with $r<R$. Assume that $\partial E_0$ is smooth in a neighbourhood of $\partial B^{n+1}_{r}(p)$ and that it intersects $\partial B^{n+1}_{r}(p)$ transversely; let $T_0$ denote (the $(n-1)$-dimensional submanifold) $\partial E_0 \cap \partial B^{n+1}_{r}(p)$.

There exists a set $E$, with finite perimeter in $B^{n+1}_{R}(p)$, that coincides a.e.~with $E_0$ in $B_R^{n+1}(p) \setminus B_r^{n+1}(p)$, that is a minimiser of $J_\lambda$ in $B_r^{n+1}(p) \subset B^{n+1}_{R}(p)$, and with the following properties:
\begin{itemize}
 \item there exists $\Sigma \subset B_r^{n+1}(p)$, closed in $B_r^{n+1}(p)$, with $\text{dim}_{\mathcal{H}}(\Sigma) \leq n-7$ such that $\big(\overline{\partial^* E} \cap B_r^{n+1}(p)\big) \setminus \Sigma$ is a smoothly embedded hypersurface with mean curvature $\lambda \nu_E$, where $\nu_E$ is the inward unit normal to $E$; more precisely, $\Sigma = \emptyset$ if $n\leq 6$, and $\Sigma$ is discrete if $n=7$.
 \item $\overline{\partial^* E} \cap \partial B^{n+1}_{r}(p) = T_0$.
\end{itemize}
\end{thm}

In the more general formulation that we will provide with Theorem \ref{thm:variational}, both smoothness and tranversality conditions will be removed (see also Remark \ref{oss:introtheorem}).   

\medskip

The `boundary condition' in Theorem \ref{thm:variational_simplified} is set by prescribing the coincidence a.e.~with a reference set $E_0$ (the condition $r<R$ provides an annulus in which $E_0$ is non-trivial). The submanifold $T_0$ acts as prescribed boundary condition for the hypersurface that we seek. The last conclusion of the theorem states that the solution does not touch $\partial B^{n+1}_{r}(p)$ except at $T_0$. So 
$\overline{\partial^* E}\cap B_r^{n+1}(p) \setminus \Sigma $ is a smooth hypersurface with boundary in the open set $B_R^{n+1}(p) \setminus \Sigma$. 
(Since $T_0$ is smooth, $\Sigma$ does not accumulate onto $T_0$ by Allard's boundary regularity theorem, \cite{All_boundary}; this property is not needed in our forthcoming arguments.) 

While the existence of a minimiser follows for any $\lambda$, the condition $\lambda <\frac{n}{r}$ is essential for the last conclusion of Theorem \ref{thm:variational_simplified}, as well as for the verification of the prescribed mean curvature condition. We will discuss this with examples in Remark \ref{oss:example_ball}; when $\lambda > \frac{n}{r}$ the hypersurface may in fact touch $\partial B_r^{n+1}(0)$ away from its boundary $T_0$. 

Theorem \ref{thm:variational_simplified} (and Theorem \ref{thm:variational} below) and its proof are close in spirit to the results in Duzaar--Fuchs \cite{DuFu} (and Duzaar \cite{Duz}). We highlight that our last conclusion in Theorem \ref{thm:variational_simplified} is sharper than the corresponding statement in \cite{Duz, DuFu}, since we are able to rule out any interior touching of the solution with the ``obstacle'' $\partial B^{n+1}_{r}(p)$ in which the boundary condition $T_0$ lies (the only touching is the necessary one at $T_0$ itself). The results in \cite{Duz, DuFu}, while establishing the validity of the CMC condition, would only prevent touching of the solution with larger spheres. The sharper conclusion we obtain is ultimately due to our use of the regularity theory for stable CMC hypersurfaces developed in \cite{BeWi1, BeWi2} (with the sheeting theorem therein being the key ingredient in our proof). The same reasoning that we employ to that end (see Lemma \ref{lem:obstacle} and the discussion preceding it) can be applied to sharpen the corresponding conclusion in \cite{DuFu} (where the relevant class is that of integral currents, rather than boundaries of Caccioppoli sets).
\medskip

We are now ready to present an outline of the proof of Theorem \ref{thm:main_gendim}, setting $p=0$. By fairly standard arguments, there exists a sufficiently small ball centred at $0$, which we denote by $B_{2R}(0)$, such that $E$ is the unique minimiser of $J_\lambda$ in $B_R(0) \subset B_{2R}(0)$, and with the further requirements that $\lambda<\frac{n}{R}$ and that $\partial E$ meets $\partial B_R(0)$ smoothly and transversely.

Then we perturb $E$ towards its interior (keeping it fixed outside an annulus that contains $\partial B_R(0)$) and use the resulting set as `boundary condition' in $B_{2R}(0) \setminus B_R(0)$ for a CMC Plateau problem. The perturbation is indexed on $j$ and tends to the identity as $j\to \infty$, and we denote the deformed set by $E_j \subset E$. For each $j$ we find a minimiser of $J_\lambda$ with said boundary condition; note that Theorem \ref{thm:variational_simplified} applies here. Theorem \ref{thm:main_gendim} follows by showing the existence of a sufficiently small ball centred at $0$ in which, for all sufficiently large $j$, $\overline {\partial^* E_j}$ are smooth. Arguing by contradiction, we assume the existence of singular points $p_j \in \overline{\partial^* E_j}$, $p_j \to 0$. If the condition $p_j \not = 0$ is valid (for all sufficiently large $j$) then we dilate $E_j$ around $0$ by scaling $B_R(0)$ to $B_{\frac{R}{|p_j|}}(0)$. Using \cite{HS}, we check that the limit of these rescalings of $\overline{\partial^* E_j}$ has to be either one of the leaves of the Hardt--Simon foliation, or the tangent cone $C$ to $\overline{\partial^* E}$ at $0$: in either case we find a contradiction to the smoothness respectively of the leaves, or of the cone (at points at distance $1$ from the origin).

Therefore we have to establish the condition $p_j \not = 0$. By construction $E_j \subsetneq E$ and both boundaries are hypersurfaces with the same scalar mean curvature, and with mean curvature vectors both pointing inwards. We thus show that the inclusion is strict everywhere by proving an instance of a singular maximum principle for CMC hypersurfaces, see Proposition \ref{Prop:sing_max_princ} below. Its proof (by contradiction) relies on a linearisation argument that yields a non-trivial Jacobi field on the cone $C$ (an analogous argument appears in \cite{HS} in the minimal case), combined with Simon's result \cite{Sim_unique_tan}, which gives a quantitative decay of $\overline{\partial^* E}$ towards $C$ at small scales. The resulting behaviour of the Jacobi field is in contradiction with the ones that are known (\cite{CaHaSi}) to be permitted by the stability of the cone (stability follows from the minimising condition for $E$).

\medskip

\noindent \textbf{Acknowledgements}. The material in this work overlaps partly with the PhD Thesis of K.~L., who would thus like to thank University College London for the stimulating environment provided during the years spent there as a graduate student. K.~L.~would also like to thank Kobe Marshall-Stevens for many enlightening discussions. The final stages of this work were completed while C.~B.~was in residence at the Simons Laufer Mathematical Sciences Institute (formerly MSRI) in Berkeley, California, during the Fall 2024 semester, supported by the National Science Foundation under Grant No.~DMS-1928930.  The authors wish to thank the anonymous referees for constructive comments and suggestions.

\section{Prescribed CMC Plateau problem}
\label{variational_setup}

In the following we denote by $B_R$ the open ball $B^{n+1}_R(0) \subset \R^{n+1}$. 
Let $E_0$ be a set of finite perimeter in $B_2$, that is, $E_0 \subset B_2$ is measurable and the perimeter of $E_0$ in $B_2$ is finite,
\[\text{Per}_{B_2}(E_0) =\sup \Bigg\{\int_{E_0} \text{div}T \: d\mathcal{L}^{n+1} \, : T\in C^1_c(B_2;\R^{n+1}), \sup|T| \leq 1 \Bigg \} <\infty,\] where $\mathcal{L}^{n+1}$ denotes the Lebesgue measure on $\R ^{n+1}.$ This is equivalent to the requirement that the characteristic function $\chi_{E_0} \in \text{BV}(B_2)$, that is, the distributional gradient $D \chi_{E_0}$ is a vector valued Radon measure with finite total variation in $B_2$.
 
 \medskip
 
For $\lambda \geq 0$ we will be interested in the following energy, defined on the class of sets of finite perimeter in $B_2$ that coincide with the given $E_0$ in $B_2 \setminus B_1$:
\[J_\lambda(E) = \text{Per}_{B_2}(E) - \lambda |E|, \]
where $|E|=\mathcal{L}^{n+1}(E)=\mathcal{H}^{n+1}(E)$ is the $(n+1)$-volume of the Caccioppoli set $E \subset B_2$. (The Lebesgue measure $\mathcal{L}^{n+1}$ agrees with the Hausdorff measure $\mathcal{H}^{n+1}$ in $\mathbb{R}^{n+1}$.) This class is non-empty, since $E_0$ is one such set, and $J_\lambda(E_0)<\infty$, hence it makes sense to seek a minimiser of $J_\lambda$ in this class.

\begin{lem}
\label{lem:Plateau_existence}
There exists a minimiser $F$ of $J_\lambda$ in the class of sets with finite perimeter that coincide with the given $E_0$ in $B_2 \setminus B_1$.
\end{lem}

\begin{proof}
We will use the direct method.
Let $E_j$, for $j\in \N \setminus \{0\}$, be a minimising sequence (of sets in the admissible class), that is 
\[\lim_{j\to \infty} J_\lambda(E_j) = \inf \big\{J_\lambda(E): \chi_E \in \text{BV}(B_2), \; \chi_E |_{B_2 \setminus B_1} = \chi_{E_0}|_{B_2 \setminus B_1} \big\}.\]
For all sufficiently large $j$ we must then have 
\[J_\lambda(E_j) =  \text{Per}_{B_2}(E_j) - \lambda|E_j| \leq J_\lambda(E_0)+1 = \text{Per}_{B_2}(E_0) - \lambda |E_0| +1 , \]
from which 
\[ \text{Per}_{B_2}(E_j) \leq \text{Per}_{B_2}(E_0) - \lambda |E_0| +\lambda|E_j| +1 \leq \text{Per}_{B_2}(E_0)+\lambda |B_2| +1.\]
Therefore $\text{Per}_{B_2}(E_j)$ are uniformly bounded above and there exist (by BV compactness) a set of finite perimeter $F$ in $B_2$ and a subsequence (that we do not relabel) $E_j$ such that $\chi_{E_j} \to \chi_F$ in BV($B_2$). In particular, $\chi_{E_j} \to \chi_F$ in $L^1(B_2)$, so that $|E_j|\to |F|$; moreover, by the hypothesis that $E_j=E_0$ on $B_2 \setminus B_1$, we have also that $F=E_0$ on $B_2 \setminus B_1$. The lower semi-continuity of perimeters then gives $J_\lambda(F) \leq \liminf_{j\to \infty} J_\lambda(E_j)$, therefore $F$ minimises $J_\lambda$ in the admissible class.
\end{proof}

The energy $J_\lambda$ is relevant in many variational problems. The geometric significance of $J_\lambda$ lies in the fact that it should select, as its critical points, sets whose boundary is a hypersurface with constant mean curvature $\lambda$. With the set up above, we are using $E_0$ to prescribe a boundary condition (in the sense of the Plateau problem). If $\partial E_0$ is smooth and intersects $\partial B_1$ transversely, then the set up amounts to fixing $\partial E_0 \cap \partial B_1$ as $(n-1)$-dimensional boundary data, and looking for a ($n$-dimensional) CMC hypersurface-with-boundary, with mean curvature $\lambda$, and whose boundary is $\partial E_0 \cap \partial B_1$. The hope is to obtain this hypersurface-with-boundary as $\partial F \setminus \big(\partial E_0 \cap (B_2 \setminus \overline {B_1})\big)$ (if $\partial F$ is smooth). 

\begin{oss}
If $\lambda<0$ and $F$ is a minimiser of $J_{|\lambda|}$ in $B_1\subset B_2$, then $U\setminus F$ is a minimiser of $J_\lambda$ in $B_1\subset B_2$ (and vice versa), so we only treat the case $\lambda \geq 0$ (and all results extend in a straightforward manner to $\lambda<0$). This follows from the fact that complementary sets have the same perimeter (in an open set).
\end{oss}

A well-known consequence of the minimising property is that the integral varifold $V$ (in $B_2$) defined by 
\[V=\big|\partial^* F \setminus \big(\partial^* E_0 \cap (B_2 \setminus \overline{B_1})\big)\big|\]
(the notation $|\quad|$ denotes the multiplicity-$1$ varifold associated to a rectifiable set) has first variation in $B_1$ represented by the vector-valued measure 
\[\lambda (\mathcal{H}^n \res (\partial^* F \cap B_1)) \nu_F,\]
where $\nu_F$ is the (measure theoretic) inward unit normal ($\mathcal{H}^n$-a.e.~well-defined on $\partial^* F$). Indeed, given any vector field $X \in C^1_c(B_1;\R^{n+1})$, we can consider, for $\delta>0$ sufficiently small, the one-parameter family of diffeomorphisms $\Phi_t=Id + t X$ for $t\in (-\delta, \delta)$. For every such $t$, we have $\Phi_t=Id$ on $B_2 \setminus B_1$ and therefore the set $\Phi_t(F)$ remains in the admissible class for every $t$. The image of $V$ under $\Phi_t$ is $\big|\partial^* \Phi_t(F) \setminus \big(\partial^* E_0 \cap (B_2 \setminus \overline{B_1})\big)\big|$.

This permits to write the stationarity condition for $V$ with respect to the energy $J_\lambda$, which gives (see e.g.~\cite[Chapters 17 and 19]{Mag})
\[\int \text{div}_{\partial^* F} X \,dV + \lambda \int (\nu_F \cdot X)\, dV =0\]
and the desired conclusion. The candidate $V$ thus has the correct mean curvature in $B_1$.

\begin{oss}
The notation $|\quad|$ has been (and will be) employed to denote the $(n+1)$-volume when the argument is a Caccioppoli set (as in $|E|$ above), and to denote the multiplicity-$1$ ($n$-dimensional) varifold associated to an $n$-dimensional rectifiable set (as for $V$ above). The context and the different character of the argument should avoid any confusion.
\end{oss}

Next we are going to examine when it is possible to conclude this same condition away from the prescribed boundary: the missing analysis at this stage is the behaviour at points that potentially lie on $\partial B_1$ but are not part of the prescribed boundary. We begin by pointing out that, if the vector field $X$ is non-zero somewhere on $\partial B_1$, then the above argument breaks down, since a one-parameter family of diffeomorphisms with initial speed $X$ may map $F$ to a set that is not in the admissible class (no matter how small $\delta$ is). In fact, the minimiser may just fail to have mean curvature $\lambda$ when $\lambda>n$, as the following examples show.

\begin{oss}
\label{oss:example_ball}
Let $H$ be the half-space $\{x_{n+1} < 0\}$ and $E_0 = H\cap B_2$. Then for any given $\lambda>n$ the minimisation procedure fails to produce a set whose boundary is a CMC hypersurface-with-boundary with mean curvature $\lambda$ and boundary condition $\partial H \cap \partial B_1$. (In fact, the unique minimiser $F$ is given by $E_0 \cup B_1$ for all $\lambda\geq n$.) To see that, we observe that, for any given possible value $v  \in \Big[\frac{|B_1|}{2}, |B_1|\Big]$, the (unique) perimeter-minimiser with volume $v$ in $B_1$, that coincides with $E_0$ in $B_2 \setminus B_1$, is given by the set $E_0 \cup E_v$, where $E_v$ is the ball of radius $r$ centred at the point $(0, \ldots, 0, -\sqrt{r^2 - 1})$, where $r\geq 1$ is chosen so that $|E_v \cap B_1|=v$. Similarly, for any given possible value $v\in |E \cap B_1| \in \Big[0, \frac{|B_1|}{2}\Big]$, the perimeter-minimiser with volume $v$ in $B_1$, and that coincides with $E_0$ in $B_2 \setminus B_1$, is given by the set $E_0 \setminus \tilde{E}_v$, where $\tilde{E}_v$ is the ball of radius $r$ centred at the point $(0, \ldots, 0, \sqrt{r^2 - 1})$, where $r\geq 1$ is chosen so that $|\tilde{E}_v \cap B_1|=|B_1|-v$.
The minimisation property just claimed is checked by a calibration argument, using the fact that $\partial E_v \cap B_1$ (and, similarly, $\partial \tilde{E}_v \cap B_1$) is a CMC graph on $B^n_1 \subset \R^n \equiv \R^n \times \{0\}$. (See e.~g.~\cite[Appendix B]{BeWi1}.) With this understood, the minimiser of $J_\lambda$ (for any $\lambda$) has to be one of the minimising sets that have been exhibited for each possible value of $v$. Each of these minimisers has scalar mean curvature in $[-n,n]$ (away from $B_2 \setminus \overline{B_1}$). Hence for any $\lambda>n$ the minimisation procedure will not produce the desired CMC hypersurface of mean curvature $\lambda$. (By direct computation, one can check that the lowest value of $J_\lambda$ for $\lambda>n$ is attained by $E_0 \cup B_1$.)

In the case $\lambda = n+1$ one can alternatively see that the minimiser is $E_0 \cup B_1$ by arguing as follows. Given any Caccioppoli set $D$ that coincides with $E_0$ in $B_2 \setminus B_1$, consider the $(n+1)$-current $C=\llbracket E_0 \cup B_1 \rrbracket - \llbracket D \rrbracket$. Denoting by $\iota_T$ the interior product with $T$, we define the $n$-form $\beta = \iota_T ( dx^1 \wedge \ldots \wedge dx^{n+1})$, with $T=(x_1, \ldots, x_{n+1})$. Then $d \beta = (\text{div} T)dx^1 \wedge \ldots \wedge dx^{n+1} = (n+1) dx^1 \wedge \ldots \wedge dx^{n+1}$. We note that $C$ is supported in $\overline{B_1}$, so it can act on $d\beta$ (by introducing a cut off function that is $1$ on $\overline{B_1}$ and vanishes outside $B_2$). Then the equality $C(d \beta) = (\partial C)(\beta)$ gives $\partial \llbracket E_0 \cup B_1 \rrbracket (\beta) - (n+1) |E_0 \cup B_1|=\partial \llbracket D \rrbracket(\beta) - (n+1) |D|$. Finally we note that $\partial \llbracket E_0 \cup B_1 \rrbracket (\beta) = \text{Per}_{B_2} (E_0 \cup B_1) - \text{Per}_{B_2 \setminus \overline{B_1}} H + \mathcal{H}^n(\partial B_2)$, while $\partial \llbracket D \rrbracket(\beta) \leq  \text{Per}_{B_2} (D) - \text{Per}_{B_2 \setminus \overline{B_1}} (H)+ \mathcal{H}^n(\partial B_2)$, which gives that $J_{n+1}(E_0\cup B_1) \leq J_{n+1}(D)$, that is, $E_0 \cup B_1$ is a minimiser. In fact, the inequality is not strict if and only if $\partial^* D \setminus (B_2 \setminus \overline{B_1})$ is a.e.~orthogonal to $T$ and contained in $\partial B_1$, which shows that $E_0\cup B_1$ is the unique minimiser.  
\end{oss}
 
Before proceeding further we set up some notation. 
The integral $(n+1)$-current $\llbracket E_0 \rrbracket$ in $B_2$ admits a well-defined (outer) slice $\langle \llbracket E_0\rrbracket, |x|=1^+\rangle= - \partial\llbracket E_0 \cap (B_2\setminus \overline{B_1}) \rrbracket + (\partial \llbracket E_0 \rrbracket) \res (B_2 \setminus \overline{B_1})$. (See e.g.~\cite[Section 2.5]{GMS}.) This (outer) slice also coincides with $\langle \llbracket F\rrbracket, |x|=1^+\rangle$. Let $T_0$ denote the $(n-1)$-dimensional current 
\[T_0=-\partial \langle \llbracket E_0\rrbracket, |x|=1^+\rangle = -\partial \Big((\partial \llbracket E_0 \rrbracket) \res (B_2 \setminus \overline{B_1})\Big),\]
then the Plateau problem under consideration seeks an integral $n$-current with boundary $T_0$. Note that $\partial \llbracket F \rrbracket = \partial \llbracket F \cap B_1 \rrbracket + \partial \llbracket E_0 \cap (B_2 \setminus B_1) \rrbracket$ so 
\[S:=\partial \llbracket F \cap B_1 \rrbracket - \langle \llbracket F\rrbracket, |x|=1^+\rangle = \partial \llbracket F \rrbracket- (\partial \llbracket E_0 \rrbracket) \res (B_2 \setminus \overline{B_1})\]
has boundary $\partial S = T_0$. The integral $n$-current $S$ is our candidate (hypersurface-with-boundary) solution to the Plateau problem. We let 
\[\mathcal{S}=\partial^* F \setminus  \big(\partial^* E_0 \cap (B_2 \setminus \overline{B_1})\big),\]
then $S = (\mathcal{S}, 1, -\star \nu_F)$, where $\star$ is the Hodge star (so $\nu_F \wedge \star \nu_F$ gives the positive orientiation of $\R^{n+1}$) and $\nu_F$ is the unit inward (measure theoretic) normal for $F$ on its reduced boundary. Also note that $V = \underline{v}(\mathcal{S}, 1)$ is the associated varifold (with notation from \cite{Simon_notes}).
 
\medskip

We turn our attention to the analysis of the first variation (with respect to $J_\lambda$) of $V$ on $B_2\setminus \text{spt}T_0$. Combining Lemma \ref{lem:Plateau_existence} with Lemmas \ref{lem:sign_first_var}, \ref{lem:stationarity}, \ref{lem:obstacle} below, we will in particular prove the following overall result.

\begin{thm}
\label{thm:variational}
With the above setting and notation, let $\lambda \in (0, n)$. In the class of sets with finite perimeter that coincide with the given $E_0$ in $B_2 \setminus B_1$ there exists a minimiser $F$ of $J_\lambda$, and there exists a set $\Sigma \subset B_1$ with $\text{dim}_{\mathcal{H}} \Sigma \leq n-7$, such that $(\text{spt}V  \setminus \text{spt}T_0) \setminus \Sigma$ is a smoothly embedded CMC hypersurface with mean curvature vector $\lambda \nu_F$. If $n=7$, more precisely, $\Sigma$ is made of isolated points (possibly accumulating onto $\text{spt}T_0$). Moreover, $\text{spt}V  \setminus \text{spt}T_0 \subset B_1$. 
\end{thm}

\begin{oss}
\label{oss:introtheorem}
By scaling and translating, the theorem can be stated replacing $B_1$, $B_2$ and $(0,n)$ respectively with $B_r^{n+1}(p)$, $B_{2r}^{n+1}(p)$, $(0,\frac{n}{r})$. Moreover, the role of $B_{2r}^{n+1}(p)$ is only to provide an annulus in which $E_0$ is non-trivial, so $2r$ can be replaced by any radius $R>r$. Theorem \ref{thm:variational_simplified} is thus a special case of Theorem \ref{thm:variational}, and in the case of Theorem \ref{thm:variational_simplified} the accumulation of $\Sigma$ onto $T_0$ is ruled out by \cite{All_boundary}. We also recall that, as well as the varifold $V$, we can associate to the minimiser $F$ an integral $n$-current $S$ such that $\partial S = T_0$ (see above for the definition of $S$).
\end{oss}

Our first result on the first variation (with respect to $J_\lambda$), Lemma \ref{lem:sign_first_var}, is valid for any $\lambda$ and yields a sign condition and an upper bound. The analysis needs to be carried out only in a neighbourhood of an arbitrary $p\in \partial B_1 \setminus \text{spt}T_0$ (since $\text{spt}V \subset \overline{B_1}$ and we have established that the first variation is $0$ in $B_1$). This result is the analogue of \cite[Theorem 4.1]{DuFu}. Here we keep using the notation introduced above (e.g.~$\nu_F$, $V$, $\mathcal{S}$, $T_0$).

\begin{lem}
\label{lem:sign_first_var} 
Let $X \in C^1_c(B_2\setminus \text{spt}T_0; \R^{n+1})$. Then the first variation with respect to $J_\lambda$ of $V$ evaluated on the vector field $X$ (equal to the left-hand-side of the following expression) satisfies

\[\int \text{div}_\mathcal{S} X d\mathcal{H}^n \res \mathcal{S}+ \lambda \int (\nu_F \cdot X)d\mathcal{H}^n \res \mathcal{S} = \int (X \cdot N) d \mathcal{M},\]
where $\mathcal{M}$ is a positive Radon measure supported in $\partial B_1$ and $N=-\frac{x}{|x|}$ (for $x\neq 0$). Moreover, $\mathcal{M} \leq \Big(\text{div}_{\mathcal{S}} N + \lambda (\nu_F \cdot N) \Big)d\mathcal{H}^n \res(\mathcal{S} \cap \partial B_1)$ (as measures).
\end{lem}

\begin{proof}
Let $p\in \partial B_1 \setminus \text{spt}T_0$ and consider $B_r(p) \subset B_{\frac{5}{4}} \setminus \text{spt}T_0$. In the first part of the proof, we analyse the action of the first variation on a vector field of the type $\eta N$, where $\eta \in C^1_c(B_r(p))$, $\eta \geq 0$.
Let $d(\cdot)=\text{dist}(\cdot, \partial B_1)$, where $\text{dist}$ is the signed distance, taken to be positive in $B_1$ and negative in $B_2 \setminus \overline{B_1}$. Note that in any tubular neighbourhood of $\partial B_1$ we have that $d$ is smooth and its gradient is $N$. 
Given $\epsilon>0$, let $f_\epsilon: \R \to \R$ be a $C^1$ function such that $f_\epsilon\equiv 0$ on $[2\epsilon, \infty)$, $f_\epsilon\equiv 1$ on $(-\infty, \epsilon]$ and $f^\prime \leq 0$. We consider the following one-sided ($s\in [0,s_0]$, with $s_0>0$ sufficiently small, depending on $\epsilon$) one-parameter family of diffeomorphisms:
\[\phi_s(z)=z+ s\,\eta(z)(f_\epsilon\circ d)(z)\,N(z).\]
The reason for the one-sided restriction, $s\geq 0$, is that we need to ensure that we stay in the admissible class when deforming via $\phi_s$, which we check next.

Since $\partial S = T_0$, and $\text{spt}S \subset \overline{B_1}$, by the conditions on $\phi_s$ we also have $\partial (\phi_s)_\sharp S = T_0$ and $\text{spt}(\phi_s)_\sharp S \subset \overline{B_1}$. 
On one hand we have $S + \langle \llbracket F \rrbracket, |x|=1^+\rangle = \partial \llbracket F \cap B_1\rrbracket$, therefore (for any $\sigma \in [0,s_0]$)
\[(\phi_\sigma)_\sharp S +  (\phi_\sigma)_\sharp \langle \llbracket F \rrbracket, |x|=1^+\rangle =  \partial (\phi_\sigma)_\sharp\llbracket F \cap B_1\rrbracket.\]
On the other hand, letting $\Phi(s,z) = \phi_s(z)$ for $s\in [0,\sigma]$ (this is a homotopy between the identity $\phi_0$ and $\phi_\sigma$ on $B_2$) we obtain, from the homotopy formula, 
\[(\phi_\sigma)_\sharp \langle \llbracket F \rrbracket, |x|=1^+\rangle - \langle \llbracket F \rrbracket, |x|=1^+\rangle = \partial  \Big( \Phi_\sharp([0,\sigma] \times \langle \llbracket F \rrbracket, |x|=1^+\rangle)\Big) .\]
Next we check that $-\Phi_\sharp([0,\sigma] \times \langle \llbracket F \rrbracket, |x|=1^+\rangle)$ is a Caccioppoli set. Note that $\Phi(s, \cdot)$ only acts on $z\in \partial B_1$ in this case. The map $\Phi|_{[0,\sigma]\times \partial B_1} : [0,\sigma]\times \partial B_1 \to B_1$ is Lipschitz and orientation-reversing wherever its differential is injective, moreover it is injective on the set where its differential is non-degenerate. Therefore, since $[0,\sigma] \times \langle \llbracket F \rrbracket, |x|=1^+\rangle$ is a Caccioppoli set in $\R \times \partial B_1$, so is its negative pushforward (e.g.~by employing the image formula for integral currents, see e.g.~\cite[p.~149]{GMS} or \cite[26.21(2)]{Simon_notes}). We finally note that $-\Phi_\sharp([0,\sigma] \times \langle \llbracket F \rrbracket, |x|=1^+\rangle)$ is disjoint from $ (\phi_\sigma)_\sharp\llbracket F \cap B_1\rrbracket$. Indeed, $\Phi \Big([0,\sigma]\times \partial B_1\Big)$ is contained in $\big\{x\in B_1 : \big|x - \frac{x}{|x|}\big| \leq \sigma \eta(\frac{x}{|x|})\big\}$, while the image $\phi_\sigma (B_1)$ is contained in $\big\{x\in B_1 : \big|x - \frac{x}{|x|}\big| > \sigma \eta(\frac{x}{|x|})\big\}$.
We can therefore conclude that 
\[ (\phi_\sigma)_\sharp S + \langle \llbracket F \rrbracket, |x|=1^+\rangle  =  \partial \llbracket\tilde{F}_\sigma\rrbracket\]
where $\tilde{F}_\sigma$ is the Caccioppoli set \[\tilde{F}_\sigma=(\phi_\sigma)_\sharp\llbracket F \cap B_1\rrbracket - \Phi_\sharp([0,\sigma] \times \langle \llbracket F \rrbracket, |x|=1^+\rangle).\]

Recalling that $F$ and $E_0$ agree in $B_2 \setminus B_1$, and since $\tilde{F}_\sigma \subset B_1$, we set 
\[F_\sigma = \tilde{F}_\sigma \cup \big( F \cap (B_2 \setminus B_1) \big)\]
and conclude that (the following is an identity between currents in $B_2$)
\[ (\phi_\sigma)_\sharp S + (\partial \llbracket E_0 \rrbracket) \res (B_2 \setminus B_1)  =  \partial \llbracket F_\sigma \rrbracket,\]
with $F_\sigma$ a set of finite perimeter in $B_2$ that coincides with $E_0$ in $B_2 \setminus B_1$ (that is, it is in the admissible class).

\medskip

The previous conclusion permits to use the minimising property of $F$, as we are allowed to compare the energy with that of $F_\sigma$ (for any $\sigma \in [0, s_0]$ --- $s_0$ depends on $\epsilon$). 
For $\epsilon>0$ fixed, we can write (from the minimising property)
\begin{equation}
\label{eq:DuFu1}
0\leq \lim_{\sigma \to 0^+} \frac{J_\lambda(F_\sigma) - J_\lambda(F)}{\sigma}  = \int_{\mathcal{S}} \text{div}_{\mathcal{S}} \big(\eta (f_\epsilon\circ d) N \big) d\mathcal{H}^n + \lambda \int_{\mathcal{S}} \nu_F \cdot \big(\eta (f_\epsilon\circ d) N \big) d\mathcal{H}^n.
\end{equation}
This equality is justified as follows. First, as by construction
\[\text{Per}_{B_2}(F_\sigma) - \text{Per}_{B_2}(F) = \mathbb M((\phi_\sigma)_\sharp S) - \mathbb M(S),\]
we can use the well-known formula for the first variation of $n$-area, which gives the first term on the right-hand-side of \eqref{eq:DuFu1}. Next we observe that, denoting by $dx$ the $(n+1)$-form $dx^1 \wedge \ldots \wedge dx^{n+1}$ and by $x=(x_1, \ldots, x_{n+1})$, and since (by Cartan's formula, denoting by $\mathfrak{L}$ the Lie derivative) $d(\iota_x dx) = \mathfrak{L}_x dx = (n+1) dx$, we have

\[|F_\sigma| - |F| =\big(\llbracket F_\sigma\rrbracket - \llbracket F \rrbracket \big)(dx) = \frac{1}{n+1}\partial \big(\llbracket F_\sigma\rrbracket - \llbracket F \rrbracket \big)(\iota_{x}dx) = \frac{1}{n+1}\big((\phi_\sigma)_\sharp S - S\big)(\iota_{x}dx)\]
\[=\frac{1}{n+1}\partial \big(\Phi_\sharp ([0,\sigma] \times S)\big)(\iota_{x}dx) =  \big(\Phi_\sharp ([0,\sigma] \times S)\big)(dx).\]
Then by direct computation (using the image formula \cite[p.149]{GMS}, \cite[26.21(2)]{Simon_notes}, together with the fundamental theorem of calculus)
\[\frac{d}{d \sigma}\Big|_{\sigma=0^+} \big(\Phi_\sharp ([0,\sigma] \times S)\big)(dx) =\big(\Phi_\sharp (\{0\} \times S)\big)(\iota_{d\Phi(\frac{\partial}{\partial s}) }dx) = \]
\[=S(\iota_{\eta (f_\epsilon \circ d) N }dx) = -\int_{\mathcal{S}} \nu_F \cdot (\eta(f_\epsilon \circ d) N) d\mathcal{H}^n,\]
which completes the proof of \eqref{eq:DuFu1}.

The next argument follows \cite[Theorem 4.1]{DuFu} verbatim. We check that the right-hand-side of \eqref{eq:DuFu1} is independent of $\epsilon$. Indeed, for $\epsilon'<\epsilon$ we consider
\[\psi_s(z) = z + s\eta(z) \big((f_\epsilon \circ d)(z) -(f_{\epsilon'} \circ d)(z)\big) N(z).\]
This is (for $s\in (-\delta, \delta)$ with $\delta>0$ sufficiently small, depending on $\epsilon'$) a (two-sided) one-parameter family of diffeomorphisms, equal to the identity in a neighbourhood of $\partial B_1$. We can then use the vanishing of the first variation under the deformation induced by $\psi_s$, that is,
\[\int_{\mathcal{S}} \text{div}_{\mathcal{S}}\Big(\eta(z) \big((f_\epsilon \circ d)(z) -(f_{\epsilon'} \circ d)(z)\big) N(z)\Big) d\mathcal{H}^n(z)\]
\[+\lambda\int_{\mathcal{S}} \nu_F(z) \cdot \Big(\eta(z) \big((f_\epsilon \circ d)(z) -(f_{\epsilon'} \circ d)(z)\big) N(z)\Big)d\mathcal{H}^n(z)=0.\]
The linearity of divergence, scalar product and integration then implies that the right-hand-side of \eqref{eq:DuFu1} is independent of $\epsilon$.

By the sign condition in \eqref{eq:DuFu1}, and viewing the right-hand-side of \eqref{eq:DuFu1} as the action of a distribution on $C^1_c$, there exists a (positive) Radon measure $\mathcal{M}$ in $B_2$ such that the right-hand-side of \eqref{eq:DuFu1} is given by $\int \eta d\mathcal{M}$. (A priori this distribution should depend on $\epsilon$, however we have proved that the action is independent of $\epsilon$.)

On the other hand, sending $\epsilon \to 0$ on the right-hand-side of \eqref{eq:DuFu1} (denoting by $\nabla_{\mathcal{S}} = \text{proj}_{T \mathcal{S}} \nabla$ the gradient on $\mathcal{S}$, a.e.~well-defined), we obtain:
\[\int_{\mathcal{S}}  (f_\epsilon\circ d)\, \nabla_{\mathcal{S}} \eta \cdot N d\mathcal{H}^n \to \int_{\mathcal{S} \cap \partial B_1}  \nabla_{\mathcal{S}} \eta \cdot N \, d\mathcal{H}^n =0,\]
where the last equality follows from the fact that $\nabla_{\mathcal{S}} \eta \cdot N =0$ a.e.~on $\mathcal{S} \cap \partial B_1$; 
\[\int_{\mathcal{S}}  (f_\epsilon\circ d) \eta\, \text{div}_{\mathcal{S}} N   d\mathcal{H}^n \to \int_{\mathcal{S} \cap \partial B_1}  \eta \, \text{div}_{\mathcal{S}} N \, d\mathcal{H}^n;\]
\[ \int_{\mathcal{S}} \eta\nabla_{\mathcal{S}}  (f_\epsilon\circ d) \cdot N  d\mathcal{H}^n =  \int_{\mathcal{S}} \eta (f_\epsilon^\prime\circ d) |\nabla_{\mathcal{S}}d|^2  d\mathcal{H}^n \leq 0,\]
where we used $\nabla d = N$ on the support of $f_\epsilon$;
\[\int_{\mathcal{S}} \nu_F \cdot \big(\eta (f_\epsilon\circ d) N \big) d\mathcal{H}^n \to  \int_{\mathcal{S}\cap \partial B_1} \eta\, \nu_F \cdot N  \,d\mathcal{H}^n.\]
These imply (expanding the divergence in \eqref{eq:DuFu1}) 
\[\int \eta d\mathcal{M}\leq \int_{\mathcal{S} \cap \partial B_1} \eta \text{div}_{\mathcal{S}} N   d\mathcal{H}^n + \lambda\int_{\mathcal{S} \cap \partial B_1}  \eta (\nu_F \cdot N)  d\mathcal{H}^n\]
is valid for all $\eta \in C^1_c(B_r(p))$, $\eta \geq 0$, hence 
\[\hat{\mathcal{M}}=\Big(\text{div}_{\mathcal{S}} N + \lambda (\nu_F \cdot N) \Big)d\mathcal{H}^n \res(\mathcal{S} \cap \partial B_1)\]
is a (positive) Radon measure. The first variation of $V$ (with respect to $J_\lambda$) computed on the test vector field $\eta N$ can be decomposed as the sum of the first variation computed on $\eta (f_\epsilon\circ d) N$ and on $\eta (1-(f_\epsilon\circ d)) N$. The latter contribution gives $0$ since $\eta (1-(f_\epsilon\circ d)) N \in C^1_c(B_1; \R^{n+1})$. Therefore the first variation of $V$ on $\eta N$ gives just \eqref{eq:DuFu1}, that is, is given by $\int \eta d\mathcal{M}$, and we have seen that $0\leq \mathcal{M} \leq \hat{\mathcal{M}}$. 

\medskip

In the first part of the proof we analysed the action of the first variation of $V$ (with respect to $J_\lambda$) on a vector field of the form $\eta N$, for $\eta \in C^1_c(B_r(p))$, $\eta \geq 0$.
Now, in the second part of the proof, we consider instead the action on a vector field $Y \in C^1_c(B_r(p); \R^{n+1})$ such that $Y \cdot N =0$. We note that in this case we are able to consider a two-sided deformation induced by $Y$, which will lead to a vanishing condition, see \eqref{eq:angular_directions} below, rather than an inequality as in \eqref{eq:DuFu1} (where we only had a one-sided deformation at our disposal).

Let $\psi_s$ be the flow of $Y$, that is, the one-parameter (two-sided) family of diffeomorphisms obtained by solving the ODE for each trajectory, $\frac{d}{ds}\Psi(s,x) = Y(x)$, with initial condition $\Psi(0,x)=x$, and setting $\psi_s(x)=\Psi(s,x)$. Then $\psi_s(B_1) \subset B_1$ and we consider $\tilde{F}_s = \psi_s(F \cap B_1)$. These are Caccioppoli sets with support in $\overline{B_1}$ and such that $\partial^*\tilde{F}_s = \psi_s(\partial^* F)$ is a.e.~contained in $B_1$. The Caccioppoli set $F_s=\tilde{F}_s \cup \big(F\cap (B_2 \setminus B_1) \big)$ is in the admissible class. We need to show that its boundary (as a current) is $(\psi_s)_\sharp S + (\partial \llbracket E_0\rrbracket) \res (B_2 \setminus B_1)$. The immediate expression for this boundary is $(\psi_s)_\sharp (\partial \llbracket F \cap B_1 \rrbracket )+ \partial \llbracket E_0 \cap (B_2 \setminus B_1)\rrbracket.$  Recalling that $S= \partial \llbracket F\cap B_1 \rrbracket - \langle \llbracket E_0\rrbracket, |x|=1^+\rangle$ we arrive at
\[(\psi_s)_\sharp S + (\psi_s)_\sharp \langle \llbracket E_0\rrbracket, |x|=1^+\rangle + (\partial \llbracket E_0\rrbracket)  \res (B_2 \setminus B_1)- \langle \llbracket E_0\rrbracket, |x|=1^+\rangle.\]
As $\Psi(t,z)$, for $(t,z) \in [0,s] \times B_2$ is a homotopy joining the identity $\psi_0$ to $\psi_s$, we will use the homotopy formula. We note that $\Psi(t,z)=z$ in a neighbourhood of $T_0 =- \partial \langle \llbracket E_0\rrbracket, |x|=1^+\rangle$, so that $\Psi_\sharp([0,s]\times \partial \langle \llbracket E_0\rrbracket, |x|=1^+\rangle)=0$. Moreover, $\Psi([0,s]\times \partial B_1) \subset \partial B_1$, so that $\Psi_\sharp([0,s]\times \langle \llbracket E_0\rrbracket, |x|=1^+\rangle)=0$ (as an $(n+1)$-current). The homotopy formula then gives $(\psi_s)_\sharp \langle \llbracket E_0\rrbracket, |x|=1^+\rangle =\langle \llbracket E_0\rrbracket, |x|=1^+\rangle$ and therefore (the following is an identity between currents in $B_2$)
\[\partial \llbracket F_s\rrbracket = (\psi_s)_\sharp S + (\partial \llbracket E_0\rrbracket)  \res (B_2 \setminus B_1).\]
We can therefore use the minimising condition to write the standard condition for the vanishing of the first variation (with respect to $J_\lambda$) as
\begin{equation}
\label{eq:angular_directions}
\int_{\mathcal{S}} \text{div}_{\mathcal{S}}Y d\mathcal{H}^n + \lambda \int_{\mathcal{S}} \nu_F \cdot Y d\mathcal{H}^n=0.
\end{equation}

\medskip

For the third (and final) part of the proof, given an arbitrary vector field $X \in C^1_c(B_r(p);\R^{n+1})$ we write the orthogonal decomposition $X = X^T + X^N$, where $X^N=(X\cdot N) N$ and both $X^T$ and $X^N$ are $C^1_c(B_r(p); \R^{n+1})$. Then the first variation of $J_\lambda$ on $X$ is given by the sum of the two actions on $X^T$ and $X^N$. For the former, in view of \eqref{eq:angular_directions} the action is $0$.
For the latter, we have that $X^N = \eta_+ N - \eta_- N$, where $\eta_+, \eta_- \geq 0$ and $\eta_+ = (X \cdot N)^+$,  $\eta_- = (X \cdot N)^-$. By the conclusion in the first part (applied separately to $\eta_+ N$ and $\eta_- N$, using the linearity of the first variation), we then have that the action is given by $\int (\eta_+ - \eta_-) d\mathcal{M} = \int (X\cdot N) d\mathcal{M}$.  
\end{proof}

As remarked in the example given in Remark \ref{oss:example_ball}, for $\lambda>n$ one may actually have $\mathcal{M} \neq 0$. If $\lambda \leq n$, on the other hand, we obtain the following result (this is analogous to \cite[Theorem 7.1]{DuFu}).

\begin{lem}
\label{lem:stationarity}
Let $\lambda \leq n$. Then $\mathcal{M}=0$, that is, $V$ is stationary (with respect to $J_\lambda$) in $B_2 \setminus \text{spt}T_0$.
\end{lem}

\begin{proof}
We have $N = \nu_F$ a.e.~on $\mathcal{S} \cap \partial B_1$ and $\text{div}_{\mathcal{S}} N = \text{div}_{\partial B_1} N$ a.e.~on $\mathcal{S} \cap \partial B_1$. By explicit computation $\text{div}_{\partial B_1} N = -n$ (where $n$ is the mean curvature of $\partial B_1$). Then the inequality $0\leq \mathcal{M} \leq \Big(\text{div}_{\mathcal{S}} N + \lambda (\nu_F \cdot N) \Big)d\mathcal{H}^n \res(\mathcal{S} \cap \partial B_1)$ obtained in Lemma \ref{lem:sign_first_var} becomes $0\leq \mathcal{M} \leq (\lambda-n)d\mathcal{H}^n \res(\mathcal{S} \cap \partial B_1)$. Thus with $\lambda \leq n$ we must have $\mathcal{M}=0$ (and if $\lambda<n$ also $\mathcal{H}^n \big(\mathcal{S} \cap \partial B_1\big)=0$). 
\end{proof}

Having established this stationarity property, in order to obtain Theorem \ref{thm:variational} we move on to the regularity of the minimiser, focusing on the case $\lambda < n$. We note immediately that, while the regularity in $B_1$ follows from the theory of minimisers, we may a priori have that $\text{spt}V \cap \partial B_1 \neq \emptyset$, and said theory is not applicable at these points. We will instead employ the regularity theory for stable CMC (or prescribed-mean-curvature) hypersurfaces \cite{BeWi1, BeWi2}, in view of which we recall some relevant notions.

We say that $p\in \text{spt}\,V$ is a classical singularity of an integral $n$-varifold $V$ in $\mathbb{R}^{n+1}$ when there exists an open ball $B^{n+1}_r(p)$ such that $\text{spt}\,V \cap B_r^{n+1}(p)$ is equal to the union of three or more hypersurfaces-with-boundary, all having a common boundary, all having $C^{1,\alpha}$-regularity up to the boundary, and with $p$ in the common boundary, and with at least two of the hypersurfaces-with-boundary meeting transversely at $p$.

Given an integral $n$-varifold $V$ in $\mathbb{R}^{n+1}$ we denote by $\text{gen-reg}\,V$ the set of points $p$ for which there exists an open ball $B^{n+1}_r(p)$ such that $\text{spt}\,V \cap B_r^{n+1}(p)$ is either a single $C^2$ embedded disc, or the union of two (distinct) $C^2$-embedded discs that lie on one side of each other and whose intersection contains $p$.

\begin{lem}
\label{lem:obstacle}
Let $\lambda<n$ and $V$, $F$ as above. Then $\text{spt}V  \setminus \text{spt}T_0 \subset B_1$. Moreover, there exists $\Sigma \subset B_1$ with $\text{dim}_{\mathcal{H}} \Sigma \leq n-7$ such that $(\text{spt}V  \setminus \text{spt}T_0) \setminus \Sigma$ is a smoothly embedded CMC hypersurface (with mean curvature vector $\lambda \nu_F$). If $n=7$, more precisely, $\Sigma$ is made of isolated points (possibly accumulating onto $\text{spt}T_0$).
\end{lem}

\begin{proof}
If $p\in \partial B_1 \cap \text{spt}V \setminus \text{spt}T_0$ is a point in $\text{gen-reg}\,V$, then by definition there exists an embedded disc $D \subset \text{spt}V \setminus \text{spt}T_0 \subset \overline{B_1}$ of class $C^2$ with $p \in D$. The $C^2$ regularity of $D$ and the stationarity of $V$ with respect to $J_\lambda$ (Lemma \ref{lem:stationarity}) imply that $D$ is CMC with mean curvature $\lambda$. (We remark that, by Allard's regularity theorem (\cite{All_interior}) and standard elliptic PDE regularity, there exists a dense open subset of $\text{spt}\,V$ that is smoothly embedded with mean curvature $\lambda$. It follows that, in the case in which the local structure of $\text{spt}\,V$ around $p$ is the union of two distinct $C^2$ embedded discs, the $C^2$ regularity of each disc implies that both discs have mean curvature $\lambda$.) The maximum principle gives a contradiction if $\lambda < n$ (since $n$ is the mean curvature of $\partial B_1$ with respect to the inward normal to $B_1$). This means that if $\lambda <n$ then $\text{gen-reg}\,V \cap (\partial B_1 \setminus \text{spt}T_0) = \emptyset$.

In other words, $\text{gen-reg}\,V \setminus \text{spt}T_0 \subset B_1$. In the (open) ball $B_1$ we are able to use the minimising assumption to further conclude that $\text{gen-reg}\,V \setminus \text{spt}T_0$ is a $C^2$ embedded hypersurface (that is, only the first occurrence in the definition of $\text{gen-reg}\,V$ can happen). This follows e.g.~from density estimates (see e.g.~\cite[Theorem 21.11]{Mag}). The minimising assumption also implies that $\text{gen-reg}\,V\setminus \text{spt}T_0$ (as a $C^2$ embedded hypersurface) is stable with respect to $J_\lambda$.

We further note that for $p \in \partial B_1 \cap \text{spt}V \setminus \text{spt}T_0$ the varifold $V$ has a unique tangent cone at $p$, given by the hyperplane that is tangent to $\partial B_1$ at $p$, possibly counted with integer multiplicity. The existence of tangent cones, and the fact that any such cone is a stationary varifold, both follow from the monotonicity-type formula for the mass, valid thanks to the stationarity with respect to $J_\lambda$. Since $\text{spt}V \subset \overline{B_1}$, any such tangent cone must be contained in a half-space (whose boundary is the tangent to $\partial B_1$ at $p$), and thus it has to be supported on that tangent hyperplane itself (see e.g.~\cite{Simon_notes}), from which the claim follows (thanks to the constancy theorem \cite{Simon_notes}).

Finally, we note the absence of classical singularities in $\text{spt}V  \setminus \text{spt}T_0$. In $B_1$, this is a consequence of the minimising property, while at any $p \in \partial B_1 \cap \text{spt}V \setminus \text{spt}T_0$ we have proved that the tangent has to be supported on a hyperplane (which rules out that $p$ could be a classical singularity).

Having checked all hypotheses, we can now apply the sheeting results from \cite{BeWi1} or \cite{BeWi2}, namely \cite[Theorems 3.1 and 3.3]{BeWi1} or \cite[Theorems 6.2 and 6.4]{BeWi2}. We conclude that, if $p \in \text{spt}V \cap \partial B_1 \setminus \text{spt}T_0$ then $\text{spt}V$ is, in a suitable coordinate system in a neighbourhood of $p$, given by the union of (finitely many) ordered $C^2$ graphs (each giving an embedded $C^2$ disc with constant mean curvature $\lambda$), and in particular $p \in \text{gen-reg}\,V$, contradicting the earlier conclusion that $\text{gen-reg}\,V \setminus \text{spt}T_0 \subset B_1$. (Alternatively, one may directly use the maximum principle, the fact that $\partial B_1$ has mean curvature $n$, and the condition $\lambda<n$, to find a contradiction.)

We thus conclude (in a first instance) that $\text{spt}V  \setminus \text{spt}T_0 \subset B_1$. At this stage one may either use the standard regularity theory for minimisers (e.g.~\cite[Theorem 21.8]{Mag} in conjunction with standard elliptic regularity) or alternatively \cite[Corollary 2.1]{BeWi1} or \cite[Corollary 1.1]{BeWi2}, for the remaining conclusions.
\end{proof}

\begin{oss}
We expect that the same regularity conclusions should hold for $\lambda=n$, albeit with the possibility that open subsets of $\partial B_1$ may be contained in $\text{spt} V \setminus \text{spt} T_0$, as in the example of Remark \ref{oss:example_ball}.
\end{oss}

\section{Regular minimal cones, graphs, Jacobi operator}\label{Jacobi fields section}

In Section \ref{sing_max_princ} we will prove Proposition \ref{Prop:sing_max_princ}, an instance of a singular maximum principle for CMC hypersurfaces, which will then be needed in Section \ref{approx}. In this section we collect some preliminaries on stable minimal cones and their Jacobi fields that will be needed in Section \ref{sing_max_princ}.

In what follows let $C$ be a regular cone that is also minimal. We recall that the notion of regular cone means that
$C = \{r y : r\geq 0, \: y \in \Sigma \},$
where $\Sigma$ (the link of $C$) is a smooth embedded compact $(n-1)$-dimensional submanifold of the unit sphere $S^n$. The minimality condition is the vanishing of the mean curvature of $C \setminus \{0\}$ (as a submanifold of $\R^{n+1}$). (This requirement is equivalent to the minimality of $\Sigma$ as a submanifold of $S^n$, see \cite{Simons}). We first recall some facts about graphs over $C$ and their mean curvature operator. 

Let $C = \partial \llbracket E \rrbracket,$ for a set\footnote{In the forthcoming sections, any minimal regular cone $C$ will arise automatically as a boundary. However any regular minimal cone has connected link $\Sigma$ (by a standard application of the maximum principle) and using this one shows that $S^n \setminus \Sigma$ has two connected components (by Alexander's duality), thus so does $\R ^{n+1} \setminus C,$ therefore there always exists $E$ such that $C = \partial \llbracket E \rrbracket$.} of locally finite perimeter $E \subset \R ^{n+1}$.  The graph of $u \in C ^2(C_1 ; \R)$ over $C_1 = (C \setminus \{0\}) \cap B_1$ is defined to be 
\[\text{gr}_C u = \{ x + u(x) N(x) : x \in C_1\},\] 
where $N$ is the inward pointing unit normal on $C \setminus \{0\}.$ 
We will be interested in functions $u$ that satisfy the following radial decay 
\begin{equation}
\label{eq:radial decay}
    \dfrac{|u(x)|}{|x|} + |\nabla u(x)| + |x| |\nabla ^2 u (x)| \xrightarrow[|x| \to 0]{} 0,
\end{equation}
where $\nabla$ denotes the Levi-Civita connection on $C\setminus\{0\}$ with respect to the Riemannian metric induced on $C\setminus\{0\}$ by the Euclidean one in $\R^{n+1}$, and $|\cdot|$ is taken with respect to the Euclidean inner product. 

We remark that there exists $M=M_\Sigma$ such that, if
\begin{equation}
    \label{eq: embeddedness est} 
    \dfrac{|u(x)|}{|x|} + |\nabla u(x)| \leq M
\end{equation}
is valid for all $x\in C_1$ then $\text{gr}_C u$ is an embedded hypersurface, with $\{0\}=(\overline{\text{gr}_C u} \setminus \text{gr}_C u) \cap B_1$ an isolated singularity when $C$ is not a hyperplane. We will assume in this section that \eqref{eq: embeddedness est} is satisfied on $C_1$. We further note that \eqref{eq:radial decay} implies the validity of \eqref{eq: embeddedness est} for all $0< |x| < r$ for sufficiently small $r$, and therefore, after rescaling, $\tilde{u}(x)=u\big(\frac{x}{r}\big)$ satisfies $\dfrac{|\tilde{u}(x)|}{|x|} + |\nabla \tilde{u}(x)| \leq M$ on $C_1$. (This fact will be implicitly used in Section \ref{sing_max_princ}.)

Assume now that the associated current to $\text{gr}_C u$ is of the form $\partial \llbracket F\rrbracket \res B_1$ \footnote{In what follows every graph of the form $\text{gr}_C u$ will arise as a boundary of a set of finite perimeter. However, since $\text{gr}_C u$ is embedded the map $G(x) = x + u(x) N(x)$ is a diffeomorphism to its image and, since $C_1$ is a boundary, $\R ^{n+1} \setminus C_1$ has two connected components thus so does $\R ^{n+1}\setminus G(C_1)$ therefore there always exist a set $F$ such that the associated current to $\text{gr}_C u$ is of the form $\partial \llbracket F \rrbracket \res B_1$.}, where $F$ is a set of finite perimeter and that $F$ is a critical point of $J_\lambda$ thus in particular we have that
\[\dfrac{d}{dt}\bigg|_{t=0} J_\lambda  (F_t)  =0,\] 
where $F_t$ is the set of finite perimeter whose boundary is $\text{gr}_C (u+ tv)$ and $v \in C^2_c(C_1; \R).$ We recall that the mean curvature operator $\mathcal{M}_C$ of the cone is defined as follows, by defining in duality its action on $u\in C ^2(C_1 ; \R)$:
\[ \dfrac{d}{dt}\bigg |_{t=0} \mathcal{H}^n\big(\text{gr}_C (u +tv)\big) = -< \mathcal{M}_ Cu, v>_{L^2},  \] 
where $< , >_{L^2}$ denotes the $L^2$-inner product on $C\setminus\{0\}$ and $v  \in C^2_c(C_1; \R)$. The PDE that the function $u$ satisfies is given in terms of $\mathcal{M}_C$ as we prove in the following:

\begin{lem} Let $u$ and $\text{gr}_C u$ be as above then 
\begin{equation}\label{eq: PDE of u}
\mathcal{M}_C u = \lambda \text{det}(Id - u A_C),
\end{equation}
 where $A_C$ denotes the second fundamental form of $C_1$.
\end{lem} 

\begin{proof}  Let $G(x) = x + u(x) N(x)$ and consider an extension $\hat N$ of $N$ (defined in an open cone over a tubular neighbourhood of $\Sigma$ in $\mathbb{S}^n$). Then for any $v \in C^2_c(C_1 ; \R)$ we have that 
\[0 = \dfrac{d}{dt}\bigg|_{t=0} J_\lambda(F_t) = -<\mathcal{M}_C u , v >_{L^2} + \lambda \int_{{gr}_C u} v \hat N \cdot \hat \nu d \mathcal{H}^n,\] 
where $F_t$ is the associated set to $\text{gr}_C(u + tv)$, $\hat \nu$ the inward pointing unit normal of $\text{gr}_C u$ and the last term is the derivative of the volume term. Using the area formula the latter can be written as $\int_ C v \hat N \cdot \hat \nu |J_G| d \mathcal{H}^n,$
where $|J_G|$ denotes the Jacobian of $G$. Thus it suffices to compute $ \hat N \cdot \hat \nu |J_G|.$  Let $(\tau_i)$ be an orthonormal basis of $C_1$ then 
\[D_{\tau_i} G \cdot \tau_j = \delta_{ij} - u A_{ij},\]
\[D_{\tau_i} G \cdot \hat N = D_{\tau_i} u,\] 
where $D_{\tau_i} G$ denotes the differential of $G$ in the direction of $\tau_i$ and $(A_{ij})$ is the matrix that corresponds to the second fundamental form of $C_1$ with respect to the chosen basis.  Consider the matrix 
\[ B = \begin{pmatrix}
    D_{\tau_1} G \cdot \hat N, & D_{\tau_1} G \cdot \tau_1, & \dots & D_{\tau_1} G \cdot \tau_n \\ 
    \vdots & \vdots &\vdots &\vdots \\ 
    D_{\tau_n} G \cdot \hat N, & D_{\tau_n}G \cdot \tau_1, & \dots & D_{\tau_n} G \cdot \tau_n
\end{pmatrix}\]
Let $B^{(k)}$ denote the $n \times n$ minor of the matrix $B$ for $2 \leq k \leq n+1$ obtained by erasing the $k$-th column of the matrix $B$, then
\[\hat \nu = \bigg( \text{det} (Id - u A_C) \hat N + \sum_{k=2}^{n+1} (-1)^{k-1} B^{(k)} \tau_{k-1}  \bigg) |J_G|^{-1}.\]
In particular $\hat N \cdot \hat  \nu |J _G| =  \text{det}(Id - u A_C) $ and this finishes the proof.  
\end{proof}

In view of \eqref{eq: PDE of u}, we recall some properties of the operator $\mathcal{M}_C,$ referring to \cite[(2.1)]{CaHaSi} and \cite[Section 3]{CaHaSi}, whose notation we adopt here. We also refer to \cite[Lemma 2.26]{CoMi} for a proof, and to \cite{HS}, and note that due to \eqref{eq: embeddedness est} the form established for $\mathcal{M}_C$ in \cite{CoMi} is the same as in \cite{CaHaSi} or \cite{HS}. The operator
$\mathcal{M}_C$ has the form
\[\mathcal M _C u = L_C u + N\bigg(x, \frac{u}{|x|}, \nabla u\bigg) \cdot \nabla^2 u(x) + \frac{1}{|x|} P\bigg(x, \frac{u}{|x|}, \nabla u(x)\bigg),\]
where $L_C u = \Delta_C u + |A_C |^2 u$ is the Jacobi field operator of the cone $C\setminus \{0\}$, $N$ is a symmetric bilinear form, $(\cdot)$ is to be understood as the trace of the linear transformations on $T_x C_1$ associated to the bilinear forms $N, \nabla^2 u$ (equivalently, using an orthonormal basis of $T_x C_1$ to write the associated matrices $N_{ij}$ and $\nabla_{ij} u$, this is $N_{ij} \nabla_{ij}u$ with summation over repeated indices) and both $N, P$  have a $C^2$-dependency on the arguments $(x, z, p) \in C_1 \times \R \times TC_1 $. Moreover $\mathcal{M}_C$ is a quasilinear elliptic operator of order two, and for $|z|, |p| \leq 1$ we have the following inequalities at $(x, z, p)$
\begin{equation}
\setlength{\jot}{10pt}
\begin{aligned} 
\label{eq:est for M_C}
   & |N  ( x, z , p)| \leq M_\Sigma  ( |z| + |p| ), \\  
&|P (x,  z , p) | \leq M_\Sigma  ( |z| + |p| ) ^2, \\  
 &| P_z| + |P_p| + |x| (|P_{xz}| + |P_{xp}|) \leq M_\Sigma (|z| + |p|),\\ 
&|x| (|N_x| + |P_x| + |N_{xz}| + |N_{xp}|) + |N_z| + |N_p| + |N_{zz}|+ \\
&+ |N_{zp}| + |N_{pp}| + |P_{zz}| + |P_{zp}| + |P_{pp}| \leq M_\Sigma,
\end{aligned}
\end{equation}
where the subscripts denote partial differentiation and $ M_\Sigma$ is a constant that depends on the dimension $n$ and the link $\Sigma$ of the cone.

The estimates in (\ref{eq:est for M_C}) along with the radial decay assumption (\ref{eq:radial decay}) allow us to  prove that the linearisation of the PDE (\ref{eq: PDE of u}) has the following form:  

\begin{lem} \label{lin lemma} Let $u, v \in C^2(C_1; \R)$ satisfy (\ref{eq:radial decay}) and $\mathcal{M}_C u =  \lambda \:\text{det}(Id - u A_C),$ $\mathcal{M}_C v =  \lambda \: \text{det}(Id - vA_C)$.  Then $h = v- u$ satisfies the following linear PDE 
\begin{equation}
\label{eq: linearisation}
L_C h = A_1 \cdot \nabla^2 h  + \dfrac{1}{|x|} A_2 \cdot \nabla h + \dfrac{1}{|x|^2} A_3 h, 
\end{equation}
where $A_1 : C_1  \to \text{End}(TC_1), \:
A_2 : C_1  \to TC_1, \:
A_3 : C_1 \to \R$ and $A_1, A_2, A_3 \xrightarrow[|x| \to 0]{} 0$. Moreover if $u, v \in C^3(C_1; \R)$ then the coefficients of the PDE are in $C^{0, \alpha}(U; \R)$ for some $\alpha \in (0,1)$ and any $U \subset \subset  C_1 . $ 
\end{lem}

\begin{proof}We first compute the operator $\mathcal{L} $ such that $\mathcal{L} h = \mathcal{M}_C v - \mathcal{M}_C u.$ We introduce the notation $N(u) = N\big(x, \frac{u}{|x|}, \nabla u\big) $ and $P(u) = P\big(x, \frac{u}{|x|}, \nabla u\big)$ (for $N, P$ introduced above). Then, since $L_C$ is linear,
\begin{equation*}
\mathcal{M}_C v -\mathcal{M}_C u = L_C h +  N(v) \cdot \nabla^2 v - N(u) \cdot \nabla^2 u +\dfrac{1}{|x|} (P(v)  - P(u) ).
\end{equation*}

We recall the standard method to rewrite $ N(v) \cdot \nabla^2 v - N(u) \cdot \nabla^2 u$. We denote by $N_{ij}$ the components of the matrix associated to the operator $N$ (in an orthonormal basis of $T_x C_1$) and compute (with implicit summation on repeated indices)

\begin{align*}
\setlength{\jot}{10pt}
    & N_{ij}\bigg(x, \frac{v}{|x|}, \nabla v\bigg) \nabla_{ij} v - N_{ij}\bigg(x, \frac{u}{|x|}, \nabla u\bigg)\nabla_{ij}u = \\ 
&=\int_0 ^1 \dfrac{d}{dt} \Big( N_{ij}(t,u,v) (\nabla_{ij}u + t(\nabla_{ij} v - \nabla_{ij} u) \Big ) dt, 
\end{align*}
with the notation $N_{ij}(t,u,v) = N_{ij}\bigg (x,  \frac {u}{|x|} + t \frac{(v-u)}{|x|},  \nabla u + t(\nabla v - \nabla u) \bigg).$
Differentiating with respect to $t$ we get the following expression
\begin{equation*}
 \setlength{\jot}{10pt}
 \begin{aligned}
     &\bigg(\int_0^1 N_{ij}(t,u,v) dt \bigg) \nabla_{ij}h +  \bigg( \int_0^1 |x|  N_{ij, z}(t,u,v) (\nabla_{ij}u + t\nabla_{ij}h)dt  \bigg) \frac{h}{|x|^2}+ \\
     & + \bigg( \int_0^1 |x| N_{ij,p}(t,u,v)(\nabla_{ij}u + t \nabla_{ij}h)dt \bigg) \cdot \frac{ \nabla h}{|x|},
   \end{aligned}
 \end{equation*}
where $N_{ij , z} , N_{ij, p}$ denote partial differentiation of $N_{ij}$ (with respect to $z$ and $p$ respectively). 
A similar computation gives that 
\begin{equation*}
\frac{1}{|x|}(P(v) - P(u)) = \bigg( \int_0^1 P_{z}(t,u,v) dt \bigg)\frac{h}{|x|^2} + \bigg( \int_0^1 P_{p}(t,u,v) dt \bigg) \cdot \frac{\nabla h}{|x|},
\end{equation*}
where again we use the notation $P(t,u,v) = P\Big (x,  \frac {u}{|x|} + t \frac{(v-u)}{|x|},  \nabla u + t(\nabla v - \nabla u) \Big)$ and $P_z, P_p$ denote partial differentiation as above. 
Putting everything together we get that 
\begin{equation*}\mathcal{M}_C v - \mathcal{M}_Cu = L_C h+ \bar A_1 \cdot \nabla^2 h + \frac{1}{|x|} \bar A_2 \cdot \nabla h +  \frac{1}{|x|^2} \bar A_3 h,
\end{equation*}
where 
\[\bar A_1 =  \int_0^1 N_{ij}(t,u,v)dt,\] 
\[\bar A_2 = \int_0^1 |x| N_{ij , p}(t,u,v)(\nabla_{ij}u + t \nabla_{ij} h) dt + \int_0^1 P_{p}(t,u,v) dt,\] 
\[\bar A_3 = \int_0^1 |x| N_{ij,z}(t,u,v) (\nabla_{ij}u +t\nabla_{ij}h)dt + \int_0^1 P_{z}(t,u,v)dt. \]
Using the estimates in (\ref{eq:est for M_C})  we have that 
\begin{equation*}
\setlength{\jot}{10pt}
\begin{aligned}
    &|\bar A_1| \leq M_\Sigma \bigg( \frac{|u|}{|x|} + \frac{|h|}{|x|} + |\nabla u| + |\nabla h| \bigg ), \\ 
    & |\bar A_2| \leq  M_\Sigma \bigg(\frac{|u|}{|x|} + \frac{|h|}{|x|} + |x| |\nabla^2 u| + |x| |\nabla^2 h|  \bigg),\\ 
    & |\bar A_3| \leq M_\Sigma \bigg(\frac{|u|}{|x|} + \frac{|h|}{|x|} + |x| |\nabla^2 u| + |x| |\nabla^2 h|  \bigg),
    \end{aligned}
\end{equation*}
where $M_\Sigma$ denotes a constant that depends on the link $\Sigma$ of the cone $C$. Thus from (\ref{eq:radial decay}) we have that $\bar A_1, \bar A_2,  \bar A_3 \to 0$ as $|x| \to 0.$

In a similar way, we now compute (using the Jacobi formula for the derivative of the determinant)
\begin{equation*}
\begin{aligned}
    \text{det}(Id - v A_C) - \text{det}(Id - u A_C) = \int_0^1 \dfrac{d}{dt}\text{det}(Id - (u + th)A_C)dt \\
    = -\int_0^1 \text{det}\big(Id - (u+th )A_C\big) \text{tr}\Big((Id - (u+th)A_C)^{-1} h A_C \Big) dt\\
    = -\int_0^1 h\, \text{det}\big(Id - (u+th )A_C\big) \text{tr}\Big((Id - (u+th)A_C)^{-1}   A_C \Big) dt = \dfrac{1}{|x|^2}\bar A_4 h,
\end{aligned}
\end{equation*}
where
\[\bar A_4=-\int_0^1 |x|^2 \text{det}(Id - (u+th )A_C) \text{tr}((Id - (u+th)A_C)^{-1} A_C ) dt .\]
From (\ref{eq:radial decay}) we have that $Id - (u +th)A_C \to Id$ as $|x| \to 0$ thus $ \bar A_4$ converges to $0$ as $|x| \to 0$ as well. 

The statement follows by setting $A_1=-\bar A_1$,  $A_2=-\bar A_2$ and $A_3=\bar A_4 - \bar A_3$.
\end{proof}

From (\ref{eq: linearisation}) we see that $L_C$ becomes the leading term of the PDE as $|x| \to 0.$ This crucial fact will allow us, in Proposition \ref{Prop:sing_max_princ} below, to construct a non-trivial positive Jacobi field of $C$. In view of that, we recall some well-known properties of the Jacobi operator $L_C$. 

For $x \in C \setminus \{0\}$ let $r=|x|$ and $\omega = \frac{x}{|x|} \in \Sigma$  denote spherical coordinates on $C.$ Then the metric of the cone is given by $g = dr^2 + r^2 g_\Sigma$ where $g_\Sigma$ is the pull-back on $\Sigma$ of the round metric on $S^n$ (via the inclusion map). The operator $L_C$ is expressed in spherical coordinates as 
\begin{equation}
\label{L_C in spherical}
L_C f = r^{-2}L_\Sigma f + r^{1-n}\partial_r (r^{n-1}\partial_r f),
\end{equation}
where $L_\Sigma = \Delta_\Sigma + |A_\Sigma|^2$ and $A_\Sigma$ is the second fundamental form of $\Sigma$ in $S^n.$ Since $L_\Sigma$ is a linear elliptic operator on a smooth compact manifold, we consider the spectrum $\lambda_1 < \lambda_2 \leq \ldots, \to +\infty$ of $-L_\Sigma$.

The first eigenvalue $\lambda_1$ is simple and it is known from \cite{CaHaSi} that $C$ is stable if and only if 
\[\max \{-\lambda_1, 0 \} \leq \frac{(n-2)^2}{4}.\] 
In particular, if $C$ is stable (which will be the case in forthcoming sections) we define $\gamma^{\pm}= \dfrac{n-2}{2} \pm \sqrt{\frac{(n-2)^2}{4}+ \lambda_1}$ and we have $\gamma^+ \geq \gamma^- \geq 0.$ 

\begin{oss} Unless $C$ is a hyperplane, one always has $\gamma^- >0$. Indeed if $\gamma^- =0 $ then $\lambda_1 = 0$ and from the variational characterisation of the first eigenvalue of $L_\Sigma$, if we take as a test function a constant function, we get that $|A_\Sigma| \equiv 0$ thus $|A_C| \equiv 0$ and $C$ is a plane. 
\end{oss} 

Any positive solution of $L_C f =0$ is of the form (see e.g.~\cite[p.~105]{HS}, and Lemma \ref{lem: rep of J-F} below)
\begin{equation} \label{eq: pos Jacobi fields} f(r \omega) = \bigg ( \dfrac{c_1^+}{r^{\gamma^+}} + \dfrac{c_1^-}{r^{\gamma^-}} \bigg)\phi_1(\omega),
\end{equation}
where $\phi_1>0$ is the first eigenfunction of $L_\Sigma$, that is $L_\Sigma \phi_1 = -\lambda_1 \phi_1$ and $c_1^+ , c_1^-$ are non-negative constants.

\section{A singular maximum principle}
\label{sing_max_princ}

We first state and prove the following fact regarding the convergence of minimisers of $J_\lambda$. Analogous results hold (with similar arguments that require building competitors) for area-minimising currents (see e.g.~\cite[Chapter 7, Theorem 2.4]{Simon_notes}) and for perimeter minimisers or almost-minimisers (see e.g.~\cite[Theorem 21.14]{Mag}).

\begin{lem}
\label{lem:conv_of_minimisers}
For $j\in \mathbb{N}$, let $E_j$ be sets with finite perimeter in $B_2$, and let $\lambda_j, \lambda \in [0, \infty)$, with $\displaystyle\lim_j \lambda_j = \lambda$. For each $j$ we assume that $E_j$ minimises $J_{\lambda_j}$ among sets that coincide with $E_j$ in $B_2 \setminus B_1$. Let $E$ be a set with finite perimeter in $B_2$ and assume that $\llbracket E_j \rrbracket \to \llbracket E \rrbracket$ (as currents) in $B_2$. Then $E$ minimises $J_\lambda$ among sets that coincide with $E$ in $B_2 \setminus B_1$. Moreover, $|\partial^* E_j| \to |\partial^* E|$ in $B_1$ (as varifolds).
\end{lem}

\begin{oss}
\label{oss:slices_theta}
Let $D$ be a set with finite perimeter in $B_2$. The outer and inner slices $\langle \llbracket D \rrbracket, |x|=1^+\rangle$ and $\langle \llbracket D \rrbracket, |x|=1^-\rangle$ are $n$-dimensional integral currents supported in $\partial B_1$ (which is $n$-dimensional), therefore there exist integer valued $BV$-functions $\theta_D^+$ and $\theta_D^-$ such that $\langle \llbracket D \rrbracket, |x|=1^+\rangle = \theta_{D}^+ (\mathcal{H}^n \res \partial B_1) \vec{\xi}$ and $\langle \llbracket D \rrbracket, |x|=1^-\rangle = \theta_{D}^- (\mathcal{H}^n \res \partial B_1) \vec{\xi}$, where $\vec{\xi}$ is the orientation of $\partial B_1$ corresponding (in Hodge duality) to the choice of outward pointing unit normal. In fact, $\theta_D^+, \theta_D^-$ are $\{0,1\}$-valued ($\mathcal{H}^n$-a.e.~on $\partial B_1$), since $\llbracket D \rrbracket$ is the current of integration on a Caccioppoli set.
\end{oss}

\begin{proof}
We remark that $\langle \llbracket E_j \rrbracket, |x|=1^+\rangle \to \langle \llbracket E \rrbracket, |x|=1^+\rangle$ as currents (since by definition $\langle \llbracket E_j \rrbracket, |x|=1^+\rangle= - \partial\llbracket E_j \cap (B_2\setminus \overline{B_1}) \rrbracket + (\partial \llbracket E_j \rrbracket) \res (B_2 \setminus \overline{B_1})$, and $\llbracket E_j \rrbracket \to \llbracket E \rrbracket$ in $B_2$ by assumption). 

Let $F$ be a set with finite perimeter that coincides with $E$ in $B_2 \setminus B_1$. Set
\[F_j = (F \cap B_1) \cup \big(E_j \cap (B_2 \setminus B_1)\big).\]
Then $F_j \to F$ as sets of finite perimeter (when $F_j$, $F$ are sets with finite perimeter, the convergence $F_j \to F$ as sets with finite perimeter is equivalent to $\llbracket F_j\rrbracket \to \llbracket F\rrbracket$ as currents). Moreover, $\langle \llbracket F_j \rrbracket, |x|=1^+\rangle =\langle \llbracket E_j \rrbracket, |x|=1^+\rangle$ by definition of $F_j$, and $\langle \llbracket F \rrbracket, |x|=1^+\rangle =\langle \llbracket E \rrbracket, |x|=1^+\rangle$ by definition of $F$.

With notation as in Remark \ref{oss:slices_theta}, we remark that $\theta_E^+ = \theta_F^+$, $\theta_{F_j}^- = \theta_F^-$ and $\theta_{F_j}^+ = \theta_{E_j}^+$. 
Using Lemma \ref{lem:splitting_perimeter_sphere} with $E_j, F_j$ in place of $D$, we rewrite the minimising condition $J_{\lambda_j}(E_j)   \leq J_{\lambda_j}(F_j)$ in the form,

\[\text{Per}_{B_1} E_j +  \mathbb{M}\big(\langle \llbracket E_j \rrbracket, |x|=1^+\rangle - \langle \llbracket E_j \rrbracket, |x|=1^-\rangle\big)- \lambda_j \mathcal{H}^{n+1}\big( E_j \big) \leq\]
\[ \text{Per}_{B_1} F +  \mathbb{M}\big(\langle \llbracket F_j \rrbracket, |x|=1^+\rangle - \langle \llbracket F_j \rrbracket, |x|=1^-\rangle\big)- \lambda_j \mathcal{H}^{n+1}\big(F_j\big).\]
(We have used $\text{Per}_{B_1} F=\text{Per}_{B_1} F_j$ and $\text{Per}_{B_2\setminus \overline{B_1}} E_j=\text{Per}_{B_2\setminus \overline{B_1}} F_j$.) The second term on the right-hand-side is written as $\int_{\partial B_1} |\theta_{F_j}^+ - \theta_{F_j}^-|=\int_{\partial B_1} |\theta_{E_j}^+ - \theta_{F}^-|$. Since $\partial B_1$ is compact, $|\theta_{E_j}^+ - \theta_{F}^-| \leq 1$, and $\theta_{E_j}^+ \to \theta_E^+ = \theta_F^+$ pointwise $\mathcal{H}^n$-a.e.~in $\partial B_1$  (by the hypothesis $\langle \llbracket E_j \rrbracket, |x|=1^+\rangle  \to \langle \llbracket E \rrbracket, |x|=1^+\rangle$), we conclude that (by dominated convergence) $\int_{\partial B_1} |\theta_{E_j}^+ - \theta_{F}^-| \to \int_{\partial B_1} |\theta_{F}^+ - \theta_{F}^-|$. The latter is $\mathbb{M}\big(\langle \llbracket F \rrbracket, |x|=1^+\rangle - \langle \llbracket F \rrbracket, |x|=1^-\rangle\big)$. Sending $j\to \infty$ and using the lower-semi-continuity of mass and perimeter on the left-hand-side, as well as $\mathcal{H}^{n+1}(E_j) \to \mathcal{H}^{n+1}(E)$, $\mathcal{H}^{n+1}(F_j) \to \mathcal{H}^{n+1}(F)$ (implied by $E_j \to E$, $F_j \to F$),
we find
\[\text{Per}_{B_1} E +  \mathbb{M}\big(\langle \llbracket E \rrbracket, |x|=1^+\rangle - \langle \llbracket E \rrbracket, |x|=1^-\rangle\big)- \lambda \mathcal{H}^{n+1}\big( E \big) \leq\]
\[ \text{Per}_{B_1} F +  \mathbb{M}\big(\langle \llbracket F \rrbracket, |x|=1^+\rangle - \langle \llbracket F \rrbracket, |x|=1^-\rangle\big)- \lambda \mathcal{H}^{n+1}\big( F\big).\]
Adding $\text{Per}_{B_2 \setminus \overline{B_1}} E=\text{Per}_{B_2 \setminus \overline{B_1}} F$ to both sides, and using Lemma \ref{lem:splitting_perimeter_sphere} again (with $E, F$ in place of $D$), the inequality obtained becomes $J_\lambda(E) \leq J_\lambda(F)$. Therefore $E$ minimises $J_\lambda$ (among sets that coincide with $E$ in $B_2 \setminus \overline{B_1}$).

Repeating the above argument with $E$ in place of $F$ shows that we must have $\text{Per}_{B_1}E = \lim_{j\to \infty} \text{Per}_{B_1} E_j$, therefore $\|\partial^* E_j\| \to \|\partial^* E\|$ as Radon measures in $B_1$ (and, by Allard's compactness for integral varifolds, $|\partial^* E_j| \to |\partial^* E|$ in $B_1$).
\end{proof}

\begin{oss}
\label{oss:mult_1_tangents}
Assume that $E$ minimises $J_{\underline{\lambda}}$ in an open set $U$ then, given a point $x \in \overline{\partial^* E}$ and a sequence of dilations $\eta_{x,r_j}(y) = \frac{y-x}{r_j}$, $r_j \searrow 0,$ consider the sequence of Caccioppoli sets $E_j = \eta_{x, r_j}(E)$ (blow up sequence). At the same time, we may consider the sequence of varifolds $|\partial ^* E_{j}| = {\eta_{x,r_j}}_\sharp |\partial^* E|$. Standard theory (respectively of Caccioppoli sets and of varifolds, see e.g.~\cite{Mag}, \cite{Simon_notes}) guarantees that both sequences subsequentially converge. Any limit in the sense of varifolds is a so-called varifold tangent cone of $|\partial^* E|$ at $x$. Lemma \ref{lem:conv_of_minimisers} implies that any varifold tangent cone is of the form $|\partial^* E_\infty|$, where $E_\infty$ is a Caccioppoli set obtained as a (subsequential) limit of $E_j$. This follows by passing to a subsequence (still denoted by $r_j$) for which we have convergence to a varifold tangent cone, and by using Lemma \ref{lem:conv_of_minimisers} (with $\lambda_j, \lambda$ therein replaced by $r_j \underline{\lambda}$ and $0$ respectively), noting that the dilated set $E_j$ is a minimiser of $J_{r_j \underline{\lambda}}$, and letting $E_\infty$ be the Caccioppoli set to which $E_j$ converges.
In particular, any varifold tangent cone has multiplicity $1$ on its regular part.
\end{oss}

\begin{oss}
If $\lambda<n$ then for a minimiser such as $E$ (similarly for $E_j$ if $\lambda_j<n$) in Lemma \ref{lem:conv_of_minimisers}, one has $\mathcal{H}^n(\partial^* E \cap \partial B_1)=0$ (see Lemma \ref{lem:stationarity}). Therefore $\partial \llbracket E \rrbracket \res \partial B_1 =0$ and $\langle \llbracket E \rrbracket, |x|=1^+\rangle = \langle \llbracket E \rrbracket, |x|=1^-\rangle$ by \eqref{eq:diff_of_slices} (therefore the standard slice $\langle \llbracket E  \rrbracket, |x|=1 \rangle$ exists).
\end{oss}

We are now ready to prove the main result of this section, an instance of maximum principle for CMC hypersurfaces with isolated singularities.

\begin{Prop}
\label{Prop:sing_max_princ}
Let $E$ and $F$ be sets with finite perimeter in $B_2$ that minimise $J_\lambda$ with respect to their own boundary condition, assumed in $B_2 \setminus B_1$. Assume that $\overline{\partial^* E} \cap (B_1 \setminus \{0\})$ is smoothly embedded, $0 \in \overline{\partial^* E}$, and that a tangent cone to $|\partial^* E|$ at $0$ is regular (which means, it is smooth away from $0$ and has multiplicity $1$ on its regular part). Assume further that $F \subset E$ and that $0 \in \overline{\partial^* F}$. Then $E \cap B_1=F \cap B_1$.
\end{Prop}

\begin{oss}
\label{remark}
Under the assumed condition on a tangent cone, by L.~Simon's renowned result \cite{Sim_unique_tan}, $|\partial^* E|$ possesses a unique tangent cone at $0$ (which has to be the one about which the regularity and multiplicity hypotheses are made).
\end{oss}

\begin{proof}
\emph{Step 1}. We begin by proving that $\overline {\partial^* F}$ is smooth in $B_r\setminus \{0\}$ for some $r>0$. Let $\Sigma \subset \overline{\partial^* F}$ denote the singular set of $\overline{\partial^* F}$. Arguing by contradiction, assume that $x_i \to 0$, $x_i \in\Sigma$. Letting $\rho_i = |x_i|$, we consider the sequence of dilations $x \mapsto \frac{x}{\rho_i}$ and take a blow up of $F$ at $0$ by setting $F_{0,\rho_i} = \frac{F}{\rho_i}$ and taking a subsequential limit $F_0$ of  $F_{0,\rho_i}$. By the assumption that $F \subset E$ we have that $F_0 \subset E_0$, where $E_0$ is the blow up of $E$ at $0$ obtained by taking the limit for said subsequence of dilations (as remarked above, the blow up for $E$ at $0$ is independent of the sequence of dilations). The stationarity property of $F$ with respect to $J_\lambda$ translates into stationarity of $F_{0,\rho_i}$ with respect to $J_{\rho_i \lambda}$, which implies that $F_0$ is stationary for the perimeter (equivalently, $J_0$). Similarly, $E_0$ is perimeter-stationary, that is, both $|\partial^* E_0|$ and $|\partial^* F_0|$ are stationary varifolds in $\R^{n+1}$. (We remark that both $|\partial^* E_0|$ and $|\partial^* F_0|$ are non-zero, since the origin is in the support of both $|\partial^* E|$ and $|\partial^* F|$ and thus both densities are $\geq 1$ by the monotonicity formula.)

More precisely, by Lemma \ref{lem:conv_of_minimisers}, $E_0$ and $F_0$ are perimeter minimisers in any compact set $K \subset \R^{n+1}$, for their own boundary condition (assumed in the complement of $K$). Clearly, $0\in \text{spt}|\partial^* E_0|\cap \text{spt}|\partial^* F_0|$. Then the singular maximum principle \cite[Theorem A (iii)]{Ilm} implies that $\text{spt}|\partial^* E_0| = \text{spt}|\partial^* F_0|$, and thus $|\partial^* E_0| = |\partial^* F_0|$. (Alternatively, one may use the maximum principle in the form given in \cite{Sim_max_p}.)

Lemma \ref{lem:conv_of_minimisers} (see Remark \ref{oss:mult_1_tangents}) also gives that $|\partial^* F_{0,\rho_i}|$ converge (as varifolds) to $|\partial^* F_{0}|$. By the choice of dilations, and by Allard's interior regularity theorem, see \cite{All_interior}, the points $\frac{x_i}{\rho_i}$ lie in $\partial B_1$ and have density $\Theta(\|\partial^* F_{0,\rho_i}\|,\frac{x_i}{\rho_i}) \geq 1+\varepsilon_0$, where $\varepsilon_0>0$ is the dimensional constant in Allard's regularity theorem. This contradicts the hypothesis that the density of $|\partial^* E_{0}|=|\partial^* F_{0}|$ is $1$ at any point distinct from $0$ (since $|\partial^* E_0|$ is a regular cone by assumption). We have therefore established the smoothness of $\overline{\partial^* F}$ in $B_r\setminus \{0\}$ for some $r>0$.

\medskip

\emph{Step 2}. As remarked above, $|\partial^* E_0|$ is the unique tangent cone to $|\partial^* E|$ at $0$. This also implies that  $|\partial^* F_0|=|\partial^* E_0|$ is the unique tangent cone for $|\partial^* F|$ at $0$ (since, given any blow up sequence, the resulting  blow up of $F$ at $0$ is contained in $E_0$, and the maximum principle implies, as above, that the two blow up sets must coincide). In particular (see \cite[Section 7]{Sim_unique_tan}), we are able to write $\partial E \cap (B_\delta \setminus \{0\})$ and $\partial F\cap (B_\delta \setminus \{0\})$, for sufficiently small $\delta>0$, as graphs of $C^2$ functions over the common cone $C_\delta = C \cap ( B_\delta \setminus \{0\})$, where $C = \overline{\partial^* E_0}$, as follows:  
\begin{equation}
 \begin{aligned}
 \label{eq:graphicality cond}
 &\partial E \cap (B_\delta \setminus \{0\}) 
 = \text{gr}_{C_\delta} u, \text{ with } u \in C^2(C_\delta;\mathbb{R}), \\
 &\partial F \cap (B_\delta \setminus \{0\}) = \text{gr}_{C_\delta} v \text{ with } v \in C^2(C_\delta;\mathbb{R}), \\
  &\lim_{|x|\to 0 } \Big(\frac{|u(x)|}{|x|} +  |\nabla u(x)|\Big) =  0 ,\\
 &\lim_{|x|\to 0} \Big(\frac{|v(x)|}{|x|} +  |\nabla v(x)| \Big) =0.
 \end{aligned}
\end{equation}

Taking the identification of $\mathbb{R}$ with $(T C_\delta)^\perp$ so that the orientation is inward (for $E_0$), we have, in view of $E \subset F$ and the fact that $|\partial^* E|$ and $|\partial^* F|$ are stationary for $J_\lambda$,
\[u\leq v \text{ and }\mathcal{M}_C u = \lambda \text{det}(Id - u A_C), \: \mathcal{M}_C v = \lambda \text{det}(Id - v A_C).\] 
Note that due to \eqref{eq:graphicality cond} the PDE for $u$ and $v$ satisfies the estimates \eqref{eq:est for M_C} in $C_\delta$ and from standard elliptic estimates, see also \cite[Section 1]{Sim_unique_tan}, we deduce that $|x| |\nabla^2 u(x)| + |x| |\nabla^2 v(x)| \to 0 $ as $|x| \to 0$ hence the radial decay \eqref{eq:radial decay} is satisfied. In particular, we may consider $h = v-u \geq 0$ and from (\ref{eq: linearisation}) we have that $h$ satisfies the linear PDE 
\[L_C h = A_1 \cdot \nabla^2 h  + \dfrac{1}{|x|} A_2 \cdot \nabla h + \dfrac{1}{|x|^2}A_3h,\] 
where $A_1, A_2, A_3 \xrightarrow[|x| \to 0]{} 0.$ Thus for any $K \subset \subset C_\delta$ we can apply the Harnack inequality to get that 
\[\sup_K h \leq C_K \inf_K h.\] 
Hence either $h > 0$ on $\overline K$ or $h \equiv 0. $ Since $K$ is arbitrary, we must have either $h\equiv 0$ on $C_\delta$, or $h>0$ on $C_\delta$ (and $h=0$ at $0$). We will next rule out the second occurrence.

\medskip

\emph{Step 3}. The minimising property of $E_0$ implies that $C$ is a stable minimal cone and thus all positive Jacobi fields are of the form \eqref{eq: pos Jacobi fields}. To prove that $u\equiv v$, we will construct a non-existent positive Jacobi field on $C \setminus \{0\}$ under the contradiction assumption that $h >0 $ on $C_\delta$. We argue as in \cite[Lemma 1.20]{HS}.

From the property that $h \to 0$ as $|x| \to 0$ we can construct a sequence of $\rho'_j \searrow 0$ such that 
\[\sup _{C_{\rho'_{j+1}}} h< \sup_{C _{\rho'_j}}h.\]
Let $x_j$ be the points where $\displaystyle \sup_{C_{\rho'_j}}h$ is achieved and set $r_j = |x_j|.$ Then $r_j \searrow 0$ (since $r_j \in (\rho'_{j+1}, \rho'_j)$) and
$\sup_{C_{r_j}}h = \sup_{\partial C_{r_j}}h.$ 
We define $h_j(x) = h(r_j x)$, for $x \in C_{\frac{\delta}{r_j}}$ and we have that
\[\sup_{C_1} h_j = \sup _{\partial C_1}h_j.\]
Let $x'_j \in \partial C_1$ where $\displaystyle \sup_{C_1} h_j$ is achieved and set 
$f_j(x) = \dfrac{h_j(x)}{M_j} , \text{ for } x \in C_{\frac{\delta}{r_j}} ,$
where $M_j = h_j(x'_j).$ 
From the PDE for $h$ we have that $f_j$ satisfies the following PDE 
\[L_C f_j = \tilde A^{(1)}_j \cdot \nabla^2 f_j + \frac{1}{|x|}\tilde A^{(2)}_j \cdot \nabla f_j + \frac{1}{|x|^2}\tilde  A^{(3)}_j f_j, \]
where $\tilde A^{(i)}_j (x) = A_i (r_j x)$ for $x \in C_{\frac{\delta}{r_j}} $ and $i= 1,2,3.$ 

Fix a set $K \subset \subset C \setminus \{0\}$ and let $K'$ be another set with $K \subset \subset K' \subset \subset C \setminus \{ 0\}$ and $x'_j \in K'.$ Notice that, from the standard regularity theory for CMC hypersurfaces, we have that $u, v \in C^\infty$ thus the coefficients $A_i$, for $i=1,2,3$ of the PDE are in $C^{0, \alpha}(K')$ and since $\tilde A_j^{(i)}$ are rescalings of $A_i$ we have that $[\tilde A_j^{(i)}]_{\alpha; K'} \leq M_1 r_j ^\alpha$, where $M_1$ is a constant independent of $j$ and $[\cdot]_{\alpha ; K'}$ is the H{\"o}lder semi-norm in $K'$ with exponent $\alpha.$ In particular, if we combine with \eqref{eq:radial decay}, we conclude that $||\tilde A_j^{(i)} ||_{0, \alpha ; K'} \to 0$, as $j \to \infty$ for $i=1,2,3$, where $||f||_{l,\alpha; K'} = ||f||_{l; K'} + \displaystyle\max_{|\beta|=l}[D^\beta f]_{ \alpha;K'}$ denotes the H{\"o}lder norm in $C^{l,\alpha}$. Thus from the $C^{2,\alpha}$-Schauder estimates, see Theorem 6.1 of \cite{GT}, we get that
\[||f_j||_{2,\alpha ; K} \leq M_3 ||f_j||_{0;K'},\]
where $M_3$ is a constant independent of $j$. 

From the Harnack inequality on $K'$ and since $x_j \in K'$ and $f_j(x_j) = 1$ we have that $||f_j||_{0 ; K'} \leq C_{K'} \inf_{K'} f_j \leq C_{K'}$ where $C_{K'}$ is a constant that depends on $K'$. Putting everything together we get that 
\[ ||f_j||_{2, \alpha ; K} \leq M_4, \]
where $M_4$ is a constant independent of $j$ (and depending on $K'$). From Arzel\`a-Ascoli theorem, after a diagonal argument and passing to a subsequence that we still index with $j$, we have that  $f_j \xrightarrow[C^2_{\text{loc}}(C\setminus \{0\})]{}f \in C^{2,\alpha}(C \setminus \{0 \}). $
From the uniform convergence of $\tilde A^{(i)}_j$ on compact sets to zero, for $i = 1,2,3$, we get that $L_C f = 0 $
in $C \setminus \{0\}$. Furthermore (again passing to a subsequence), we have that $x'_j \to x_0 \in \partial C_1$ and so $f(x_0) =1 $. Thus from Harnack's inequality $f>0.$ 

In conclusion, we have constructed a positive solution of $L_C f =0$, defined on $C \setminus \{ 0\}$ for a stable minimal cone $C$ of $\R^{n+1}$, and satisfying
\[\sup_{C_1 } f = \sup_{\partial C_1}f.\] 
The latter contradicts (\ref{eq: pos Jacobi fields}) and thus proves that $ \partial E \cap B_\delta = \partial F \cap B_\delta.$ 

\medskip

\emph{Step 4}. Finally we show that $E \cap B_1=  F \cap B_1$. Let \[r_0 = \sup\{r :  \partial E \cap B_r = \partial F \cap B_r\}\]
and note that the set over which we take the supremum is non-empty due to the existence of $\delta$, from the previous step, and it is in fact a maximum. Assume for the contrary that $r_0 < 1$ and let $x_0 \in \partial B_{r_0}\cap \partial F \cap \partial E.$  Then by virtue of Remark \ref{oss:mult_1_tangents}, we can consider a varifold tangent cone for $|\partial^* F|$ at $x_0$, of the form $|\partial^* G|$, with $|\partial^* G|$ stationary (for the perimeter functional), and with $\text{spt}|\partial ^* G|$ contained in a half space thanks to the condition $F\subset E$ (more precisely, the half space whose boundary is the tangent plane to $|\partial^* E|$ at $x_0$).
Then from Theorem 36.5 of \cite{Simon_notes} we have that $|\partial^* G|$ is a plane hence the regularity theory implies that we can find a neighborhood $B_{\rho'}(x_0)$ where $\partial F$ is smooth and $\partial F, \partial E$ meet tangentially at $x_0.$ Since $F \subset E$ and due to the variational equations satisfied by $J_\lambda$ the mean curvature vectors point in the same direction at $x_0$ thus the standard maximum principle implies that $\partial E \cap B_{\rho'}(x_0)$ coincides with $\partial F \cap B_{\rho'}(x_0).$ In particular, since $x_0$ is arbitrary and $\partial B_r \cap \partial F$ is compact we can find $\epsilon >0$ such that $\partial E \cap B_{r_0+ \epsilon} = \partial F \cap B_{r_0+\epsilon}$ contradicting the choice of $r_0.$ Thus $r_0 =1$ and we conclude that $E \cap B_1 = F \cap B_1.$ 
\end{proof}

\section{Approximation}
\label{approx}

Lemma \ref{lem:unique_minimiser} and Theorem \ref{app thm} below will establish in particular the approximation results stated in the introduction, Theorems \ref{thm:main_dim8} and \ref{thm:main_gendim}. (One should identify $\hat{B}-p$ in Theorem \ref{thm:main_gendim} with the ball $B_R$ below.)

We assume that $E\subset \R^{n+1}$ satisfies the following properties. The topological boundary agrees with $\overline{\partial^* E}$ and $T=\partial E$ contains $0$, the hypersurface $(T \setminus \{0\}) \cap B_R$ is smooth for some $R>0$ (so the origin is an isolated singularity for $T$), $E$ minimises $J_\lambda$ in $B_R$ among Caccioppoli sets that coincide with $E$ in $B_{2R}\setminus B_R$, a tangent cone to $|\partial^* E|$ at $0$ is regular (which means, it is smooth away from $0$ and has multiplicity $1$ on its regular part).  In view of Remark \ref{remark}, $|\partial ^* E|$ thus possesses a unique tangent cone at $0.$

\begin{oss}
\label{oss:n=7}
We note that if $n=7$ these properties can be fulfilled whenever we have a Caccioppoli set that minimises $J_\lambda$ locally. To begin with, one chooses a system of coordinates centred at a singular point, and $R$ smaller than the distance of this to any other singular point (which is possible thanks to the interior regularity theory for minimisers). Moreover (again by the regularity theory) any tangent cone must be smooth away from the origin (for otherwise, the radial invariance would give a singular set of dimension at least one). Finally, any tangent cone must have multiplicity $1$ on its regular part since the rescaled varifolds $|\partial E_{\rho_i, 0}|$ converge as varifolds to $|\partial^* E_0|$ (see Remark \ref{oss:mult_1_tangents}). 
\end{oss}

It may not be true, in the above situation, that $E$ is the unique minimiser of $J_\lambda$, among Caccioppoli sets that coincide with $E$ in $B_{2R}\setminus B_R$. However, by taking a slightly smaller $R$ (which preserves all the assumptions above), we can ensure said uniqueness, thanks to a standard argument that we now recall.

\begin{lem}
\label{lem:unique_minimiser}
Let $E, T$ be as above. If $R^\prime<R$ then $E$ is the unique minimiser of $J_\lambda$ among sets that coincide with $E$ in $B_{2R} \setminus B_{R^\prime}$. (And therefore also among sets that coincide with $E$ in $B_{2R^\prime} \setminus B_{R^\prime}$.)
\end{lem}

\begin{proof}
Let $R^\prime<R$.
Clearly, $E$ minimises $J_\lambda$ in $B_{R^\prime}$ among Caccioppoli sets that coincide with $E$ in $B_{2R}\setminus B_{R^\prime}$. Assume that there exists a Caccioppoli set $E^\prime \neq E$ that minimises $J_\lambda$ in $B_{R^\prime}$ among Caccioppoli sets that coincide with $E$ in $B_{2R}\setminus B_{R^\prime}$. In particular $E^\prime$ coincides with $E$ in $B_{2R} \setminus B_R$, and on $E^\prime$ the energy $J_\lambda$ attains the same value as it does on $E$. Therefore $E^\prime$ is a minimiser of $J_\lambda$ in $B_R$, among Caccioppoli sets that coincide with $E$ in $B_{2R}\setminus B_R$. As such, its reduced boundary must enjoy the optimal regularity of minimisers, that is, $\overline{\partial^* E^\prime} \cap B_R$ is a smooth hypersurface (with mean curvature $\lambda$) away from a set $\Sigma \subset \overline{\partial^* E^\prime} \cap B_R$ with $\text{dim}_{\mathcal{H}} \Sigma \leq n-7$. We aim to prove that $\overline{\partial^* E^\prime}$ coincides with $\overline{\partial^* E}$ (which is in contradiction with $E^\prime \neq E$ and $E^\prime = E$ in $B_{2R}\setminus B_{R^\prime}$).

We define $r \leq R^\prime$ by
\[r = \inf \{t: \overline{\partial^* E^\prime} =\overline{\partial^* E} \text{ in } B_{2R} \setminus B_t.\}\]
and note that this is a minimum. 
The conclusion will follow upon establishing that $r=0$. Assume $r>0$. We remark that for $p \in \partial B_{r} \cap \overline{\partial^* E}$ we must have that there exists a unique tangent cone to $|\partial^* E^\prime|$ at $p$, and it must coincide with the hyperplane that is tangent to $\partial E$ at $p$. (This follows from $\overline{\partial^* E^\prime} =\overline{\partial^* E}$ in $B_{2R} \setminus B_r$ and the smoothness of $\overline{\partial^* E}$ around $p$.) The regularity theory implies that $\overline{\partial^* E^\prime}$ is smooth in an open ball $B^{n+1}_\rho(p)$ for some $\rho>0$. Recall however that 
\[\overline{\partial^* E^\prime} = \big(\overline{\partial^* E^\prime} \cap B_{r} \big) \cup \big(\overline{\partial^* E} \cap (B_{2R} \setminus B_{r}) \big),\]
and we have established that this is smooth in $B_\rho(p)$. Unique continuation implies that $\overline{\partial^* E^\prime} \cap B_{r}$ coincides with $\overline{\partial^* E}  \cap B_{r}$ in $B_\rho(p)$. 

As $p \in \partial B_{r} \cap \overline{\partial^* E}$ is arbitrary and $\partial B_{r} \cap \overline{\partial^* E}$ is compact, it follows that $\overline{\partial^* E^\prime}$ coincides with $\overline{\partial^* E}$ in $B_{2R} \setminus B_{r - \delta}$ for some $\delta>0$, contradicting the choice of $r$. Hence $r=0$ and $E^\prime = E$ in $B_{2R}$.
\end{proof}

\begin{oss}
By taking $R^\prime$ sufficiently small we also ensure that $\lambda<\frac{n}{R^\prime}$. Therefore, upon dilating $B_{2 R^\prime}$ to $B_2$, we have that the working assumptions stated in the next theorem are fulfilled.
\end{oss}

\begin{thm}\label{app thm}
Let $E$ a set of finite perimeter in $B_2$. Assume that $T=\partial E = \overline{\partial^* E}$ contains $0$, the hypersurface $T  \cap (B_2 \setminus \{0\})$ is smooth, $E$ is the unique minimiser for $J_\lambda$ in $B_2$ among Caccioppoli sets that coincide with $E$ in $B_{2}\setminus B_1$, $\lambda<n$.

Given $r \in (0,1)$, there exists a sequence of sets $E_j$ that have finite perimeter in $B_2$, such that: $\partial E_j \cap B_r$ is smooth for each $j$, it has constant mean curvature $\lambda \nu_{E_j}$, where $\nu_{E_j}$ is the inward unit normal to $E_j$, $E_j \subset E$, $E_j \to E$ and $\partial E_j$ converge to $\partial E$ smoothly on any $\Omega \subset \subset B_r \setminus \{0\}$.
\end{thm}

\begin{oss}
We point out that the sequence $E_j$ will be constructed without any dependence on $r$; however, we will only prove that the boundaries $\partial E_j \cap B_r$ are smooth for sufficiently large $j$, with dependence on $r.$
\end{oss}

\begin{proof}
\emph{Step 1}. The first step is to perturb the boundary condition $E$ inwards, and then use this new boundary condition to define $E_j$. The vector field $\nu_E$ is smooth in $(B_2 \setminus \{0\}) \cap \partial E$. Let $d(\cdot) = \text{dist}(\cdot,\partial E)$ be the signed distance function to $\partial E$ (taken to be positive in $E$ and negative in its complement) and consider a tubular neighborhood $\mathcal{N}_\rho$ of size $\rho>0$ around $\partial E \cap (B_{\frac{3}{2}} \setminus B_{\frac{1}{2}})$. Then the gradient of $d$ is a smooth extension of $\nu_E$ to $\mathcal{N}_\rho$. Let $\chi$ be a smooth function on $B_2$ that is equal to $1$ in $(B_{\frac{5}{4}} \setminus B_{\frac{3}{4}}) \cap \mathcal{N}_{\frac{\rho}{2}}$ and with support contained in $(B_{\frac{3}{2}} \setminus B_{\frac{1}{2}}) \cap \{|d|<\frac{3}{4}\rho\}$. Let $X=\chi \nabla d$, then $X$ extends $\nu_E$ and we may consider the flow $\phi_t(x)$ of $X$. (We view $X$ as a vector field in $B_2$.) For any $t\in [0, \delta)$, with $\delta>0$ sufficiently small, $\phi_t(E) \subset E$. By construction $\phi_t(\partial E \cap \partial B_1)$ is disjoint from $\partial E \cap \partial B_1$ for all $t\in (0, \delta)$, and $\phi_t(E) \cap (B_{\frac{5}{4}} \setminus B_{\frac{3}{4}})$ is strictly contained in $E \cap (B_{\frac{5}{4}} \setminus B_{\frac{3}{4}})$.

\medskip

The sequence $E_j$ in the statement is built with the boundary condition $E_j = \phi_{t_j}(E)$ in $B_2 \setminus B_1$, for a sequence $t_j \to 0$. Namely, from Theorem \ref{thm:variational} we may define $E_j$ to be a minimiser of $J_\lambda$ for said boundary condition; furthermore, we have that $\overline {\partial^* E_j} \cap \overline {B_1}$ intersects $\partial B_1$ only at its boundary, and is a smoothly embedded CMC hypersurface-with-boundary away from a codimension $7$ set. 
 
\medskip

\emph{Step 2}. We show first that $E_j \to E$ as $j\to \infty$ as sets of finite perimeter (therefore $\llbracket E_j \rrbracket \to \llbracket E \rrbracket $ as currents, hence $\partial \llbracket E_j\rrbracket \to \partial  \llbracket E \rrbracket$ as well). This follows from the uniqueness property of $E$, as we now show. To begin with, we have $J_\lambda(E_j) \leq J_\lambda(\phi_{t_j}(E))$ (by the minimising property of $E_j$). By smoothness of $X$, using the area formula we find that $J_\lambda(\phi_{t_j}(E)) \to J_\lambda(E)$ as $j\to \infty$. In particular, there exists a uniform upper bound for $J_\lambda(\phi_{t_j}(E))$, and thus (since $|E_j|\leq |B_2|$) a uniform upper bound for $\text{Per}_{B_2}(E_j)$. Standard BV-compactness then gives the existence of a subsequential limit $E_j\to D$ with $|E_j|\to |D|$ and (by lower semi-continuity of perimeter) $J_\lambda(D) \leq \liminf_{j \to \infty} J_\lambda(E_j)$. Recalling the previous considerations, $J_\lambda(D) \leq \liminf_{j \to \infty} J_\lambda(E_j) \leq J_\lambda(E)$. Finally, noting that $E_j \cap (B_2 \setminus \overline{B_1}) = \phi_{t_j}(E) \cap (B_2 \setminus \overline{B_1}) \to E \cap (B_2 \setminus \overline{B_1})$, we obtain that $D=E$ in $B_2 \setminus \overline{B_1}$ and therefore $D$ is a minimiser (among sets with finite perimeter that coincide with $E$ in $B_2 \setminus \overline{B_1}$). The uniqueness hypothesis on $E$ gives $E=D$.

\medskip

Next we will prove that $E_j \subset E$, for each given $j$. Considering the sets with finite perimeter $E_j \cap E$ and $E_j \cup E$, we have  $\llbracket E_j \cap E \rrbracket + \llbracket E_j \cup E \rrbracket = \llbracket E_j \rrbracket + \llbracket E \rrbracket$, so that $\partial \llbracket E_j \cap E \rrbracket + \partial \llbracket E_j \cup E \rrbracket = \partial \llbracket E_j \rrbracket + \partial \llbracket E \rrbracket$. Clearly we also have $E_j \cap E \subset E_j \cup E$. This implies that at $\mathcal{H}^n$-a.e.~$x \in \partial^* (E_j \cap E) \cap \partial^* (E_j \cup E)$ one must obtain the same half-space as the unique blow up at $x$ for both sets $E_j \cap E$ and $E_j \cup E$, and therefore the measure-theoretic outer normals are the same at $x$ for both sets. The common orientation $\mathcal{H}^n$-a.e.~gives the equality
\[\mathbb{M} \big(\partial \llbracket E_j \cap E \rrbracket\big) + \mathbb{M} \big(\partial \llbracket E_j \cup E \rrbracket\big) = \mathbb{M} \big( \partial \llbracket E_j \cap E \rrbracket + \partial \llbracket E_j \cup E \rrbracket \big),\]
and therefore
\[\mathbb{M} \big(\partial \llbracket E_j \cap E \rrbracket\big) + \mathbb{M} \big(\partial \llbracket E_j \cup E \rrbracket\big) = \mathbb{M} \big(\partial \llbracket E_j \rrbracket + \partial \llbracket E \rrbracket\big) \leq \mathbb{M} \big(\partial \llbracket E_j \rrbracket \big)+ \mathbb{M} \big(\partial \llbracket E \rrbracket\big).\]
Noting that $|E_j \cap E| + |E_j \cup E| = |E_j| + |E|$, we conclude that 
\[J_\lambda(E_j \cap E) + J_\lambda(E_j \cup E) \leq J_\lambda(E_j) + J_\lambda(E).\]
On the other hand, since $\phi_{t_j}(E) \subset E$ and $E_j$ agrees with $\phi_{t_j}(E)$ in $B_2 \setminus B_1$ we conclude that $(E_j \cap E)\cap (B_2 \setminus B_1) = E_j\cap (B_2 \setminus B_1)$ and $(E_j \cup E)\cap (B_2 \setminus B_1) = E \cap (B_2 \setminus B_1)$ thus the minimising properties of $E_j$ and $E$ imply respectively that 
\[J_\lambda(E_j \cap E) \geq J_\lambda(E_j), \, \hspace{2mm} J_\lambda(E_j \cup E) \geq J_\lambda(E).\]
Combining the inequalities obtained, we find that equalities must hold throughout, and therefore $E_j \cup E$ is a minimiser of $J_\lambda$ (among sets with finite perimeter that coincide with $E$ in $B_2 \setminus \overline{B_1}$), so that the uniqueness of $E$ gives $E_j \cup E = E$, that is, $E_j\subset E$. \footnote{We point out that the conclusion $E_j \subset E$ would follow also without the uniqueness assumption on $E$, by exploiting interior regularity for the minimiser $E_j \cup E$ in $B_1$ to conclude that $\partial^* E_j$ and $\partial^* E$ cannot intersect transversely on their regular parts, and by then applying the maximum principle and unique continuation to exclude tangential intersections.}

\medskip

\emph{Step 3}. We conclude the proof of Theorem \ref{app thm} by showing that, given any $r<1$, the sequence $\partial E_j \cap B_r$ is smooth for large $j$ (depending on $r$). To that end we will use the Hardt--Simon foliation provided by \cite[Theorem 2.1]{HS}. First note that as a consequence of Allard's interior regularity theorem, and of the smoothness of $\partial E$ away from the origin, we must then have that, for any $r<1$ and $\sigma \in (0, r)$, there is $C^{1,\alpha}$ convergence of $\partial E_j$ to $\partial E$ in $B_r \setminus B_\sigma$. By elliptic regularity, the convergence is in fact smooth, and $\partial E_j \cap (B_r \setminus B_\sigma)$ is smooth for all sufficiently large $j$, depending on the choice of $\sigma, r$.

Let $\Sigma_j$ denote the singular set of $\overline{\partial^* E_j}$ in $B_1$ (which is of dimension at most $n-7$). Let $r_0<1$ be fixed and let $p_j \in \Sigma_j \cap B_{r_0}$. In view of the previous conclusion, we must have $p_j \to 0$ as $j\to \infty$. Also we remark that, by Proposition \ref{Prop:sing_max_princ}, we must have $0\notin \overline{\partial^* E_j}$ for all $j$, so $p_j \neq 0$ for all sufficiently large $j$. We will dilate $E_j$ around $0$ by the homothety $\eta_j(x)=\frac{x}{|p_j|}$. Then $\tilde{E}_j = \eta_j(E_j)$ is a Caccioppoli set in $B_{\frac{1}{|p_j|}}$, in particular in $B_2$ for all sufficiently large $j$; moreover, the point $\tilde{p}_j = \frac{p_j}{|p_j|}$ is singular for $\overline{\partial^* \tilde{E}_j}$ and lies on $\partial B_1$. Upon extracting a subsequence that we do not relabel, we can assume that $\tilde{E}_j \to \Omega$ and $|\partial^* \tilde{E}_j|$ converge to the (stationary) integral varifold $|\partial^* \Omega|$ in $B_2$. The minimising property of $E_j$ with respect to $J_\lambda$ implies that $\Omega$ minimises perimeter in any compact set. Moreover, as $E_j \subset E$, we have $\Omega \subset E_0$, where $E_0$ is the blow up of $E$ at $0$ obtained from $\eta_j$. Then \cite[Theorem 2.1]{HS} (specifically, its final assertion) implies that either $\Omega=E_0$, or $\Omega$ belongs to the ``Hardt--Simon family'' of sets $G_s=\eta_{0,s}(G)$, where $\eta_{0,s}(\cdot)=\frac{\cdot}{s}$, $s>0$, and $G \subsetneq E_0$ has smooth minimising boundary. On the other hand, the presence of a sequence of singular points $p_j \in \partial B_1$ implies, by Allard's interior regularity theorem, that a subsequential limit $p \in \partial B_1$ of $p_j$ must occur with density $\geq 1+ \epsilon_0$ in $|\partial^* \Omega|$, contradicting the smoothness and unit density of $\partial E_0$ and of $\partial G_s$ (regardless of $s$) in a tubular neighbourhood of $\partial B_1$. The contradiction shows that $\Sigma_j \cap B_{r_0} = \emptyset$ for all sufficiently large $j$, so that $\partial E_j \cap B_{r_0}$ is a smooth hypersurface (for all sufficiently large $j$).

\end{proof}

\appendix

\section{Auxiliary results} \label{Appendix}
 
We give a proof of the following general property.
\begin{lem}
\label{lem:splitting_perimeter_sphere}
Let $D$ be a set with finite perimeter in $B_2$, then
\begin{equation}
\label{eq:diff_of_slices}
\partial \llbracket D \rrbracket \res \partial B_1 = \langle \llbracket D \rrbracket, |x|=1^-\rangle - \langle  \llbracket D \rrbracket, |x|=1^+\rangle
\end{equation}
and
\begin{equation}
 \label{eq:splitting_perimeter}
\emph{\text{Per}}_{B_2} D = \emph{\text{Per}}_{B_1} D + \emph{\text{Per}}_{B_2 \setminus \overline{B_1}}D + \mathbb{M}\big(\langle \llbracket D \rrbracket, |x|=1^+\rangle - \langle  \llbracket D \rrbracket, |x|=1^-\rangle\big).
\end{equation}
\end{lem}

\begin{proof}
To check this, we begin by recalling that for an open set $U \subset B_2$, one has $\text{Per}_U D = \mathbb{M}\big( \partial \llbracket D \rrbracket \res U\big)$, and $\mathbb{M}\big( \partial \llbracket D \rrbracket \big) = \mathbb{M}\big( \partial \llbracket D \rrbracket \res B_1 \big) + \mathbb{M}\big( \partial \llbracket D \rrbracket \res (B_2 \setminus \overline{B_1}) \big) + \mathbb{M}\big( \partial \llbracket D \rrbracket \res \partial B_1 \big)$. Therefore \eqref{eq:splitting_perimeter} follows from \eqref{eq:diff_of_slices}.

We recall that the restriction of $\partial \llbracket D \rrbracket$ to $\partial B_1$ is well-defined (since the current is normal) via the limit, for any $n$-form $\omega$ with compact support in $B_2$,
\[\big(\partial \llbracket D \rrbracket \res \partial B_1 \big)(\omega) = \lim_{h \to 0} \big(\partial \llbracket D \rrbracket\big)\big(\gamma_h(|x|-1) \,\omega \big),\]
where $\gamma_h:(-\infty,\infty) \to \mathbb{R}$ is $C^1$, is identically $1$ on $(-h,h)$, vanishes on $(-\infty,-2h)\cup (2h, \infty)$, and $\gamma^\prime\in \big[-\frac{2}{h}, 0 \big]$ on $(0,\infty)$ and $\gamma^\prime\in \big[ 0,\frac{2}{h}\big]$ on $(-\infty, 0)$. Then
\begin{equation}
\begin{aligned}
\label{eq:mass_restr_sphere}
\big(\partial \llbracket D \rrbracket \res \partial B_1 \big)(\omega) =  \lim_{h \to 0} \big( \llbracket D \rrbracket\big)\big(\gamma_h^\prime(|x|-1) d|x|\wedge\omega \big) + \lim_{h \to 0} \big( \llbracket D \rrbracket\big)\big(\underbrace{\gamma_h(|x|-1) \,d \omega}_{\to 0 \text{ as } h\to 0} \big)\\
= \lim_{h \to 0} \big( \llbracket D \rrbracket\big)\big(\gamma_h^\prime(|x|-1) d|x|\wedge\omega \big) .
\end{aligned}
\end{equation}
On the other hand, let $\gamma_h^+:(-\infty,\infty) \to \mathbb{R}$ be $C^1$, identically $0$ on $(-\infty,0)$, and equal to $1-\gamma_h$ on $[0, \infty)$. Let $\gamma_h^-:(-\infty,\infty) \to \mathbb{R}$ be defined by $\gamma_h^+(s) = \gamma_h^-(-s)$. Note that $\gamma_h^+ + \gamma_h^- + \gamma_h =1$. Then
\[\langle \llbracket D \rrbracket, |x|=1^+\rangle (\omega) =  -\partial  \Big(\llbracket D\rrbracket \res \{|x|>1\}\Big) (\omega) + \Big( \partial \llbracket D \rrbracket  \res \{|x|>1\} \Big) (\omega)=\]
\[-\lim_{h\to 0}\llbracket D \rrbracket (\gamma_h^+(|x|-1) d \omega) + \lim_{h\to 0}\llbracket D \rrbracket (d(\gamma_h^+(|x|-1) \omega)) =  \lim_{h\to 0}\llbracket D \rrbracket \big((\gamma_h^+)^\prime(|x|-1) d|x| \wedge \omega)\big) \]
and similarly
\[\langle \llbracket D \rrbracket, |x|=1^-\rangle (\omega) = - \lim_{h\to 0}\llbracket D \rrbracket \big((\gamma_h^-)^\prime(|x|-1) d|x| \wedge \omega)\big). \]
Therefore
\[\Big(\langle \llbracket D \rrbracket, |x|=1^+\rangle - \langle \llbracket D \rrbracket, |x|=1^-\rangle \Big) (\omega) =-\lim_{h\to 0}\llbracket D \rrbracket \big((\gamma_h)^\prime(|x|-1) d|x| \wedge \omega)\big), \]
which, jointly with \eqref{eq:mass_restr_sphere}, gives \eqref{eq:diff_of_slices}.
\end{proof}

We provide the details regarding the positive solutions to the linear elliptic PDE $L_C f=0$, which is crucial in the proof of Proposition \ref{Prop:sing_max_princ}. 

\begin{lem}\label{lem: rep of J-F} Let $C$ be a regular stable minimal $n$-cone in $\R ^{n+1}$. Then every positive solution of $L_C f =0$ is of the form 
\begin{equation*}
f(r \omega) = \bigg ( \dfrac{c_1^+}{r^{\gamma^+}} + \dfrac{c_1^-}{r^{\gamma^-}} \bigg)\phi_1(\omega),
\end{equation*}
where $\phi_1>0$ is the first eigenfunction of $L_\Sigma$, and $c_1^+ , c_1^-$ are non-negative constants.  
\end{lem}

\begin{proof}
We assume first that the cone is strictly stable thus $\gamma^+>\gamma^-$. Consider the eigenvalues of the operator $-L_\Sigma$, 
\[\lambda_1 < \lambda_2 \leq \lambda_3 \dots \to \infty\] and
let $(\phi_j)$ be an orthonormal basis of $L^2(\Sigma)$ such that $\phi_j$ is an eigenfunction of $\lambda_j.$ Recall that $\phi_1 >0$ and $\lambda_1$ is a simple eigenvalue. 

For any $r>0$ the function $f(r, \cdot)$ (on $\Sigma$) is of the form $\sum_{j=1}^\infty a_j(r) \phi_j (\omega).$ Thus in order to solve $L_C f = 0$ we write $L_C$ in spherical coordinates and from (\ref{L_C in spherical}) we get, after solving the corresponding ODE for $a_j$, that $a_j(r) = c_j ^+ r^{-\gamma_j^+} + c_j^- r^{-\gamma_j^-}$, where $\gamma_j^\pm = \dfrac{n-2}{2} \pm \sqrt{\dfrac{(n-2)^2}{4}  +\lambda_j}$ and $c_j^\pm$ are constants. Thus
 \[f(r \omega) = \sum_{j=1}^\infty c_j^\pm r^{-\gamma_j ^\pm}\phi_j(\omega).\]
\medskip

We want to prove that $c_j^\pm = 0$ for all $j \geq 2.$ Since $L_C f = 0 $ and $f > 0$ from Harnack's inequality on $K_1 = C \cap (\overline {B_2} \setminus B_{\frac{1}{2}})$, Corollary 8.21 of \cite{GT}, we have that $\displaystyle \sup_{K_1} f \leq C_{K_1} \inf_{K_1} f, $ where $C_{K_1}$ is a constant that depends on $K_1$ and the operator $L_C$. Let now $K_s = C \cap (\overline{B_{2s} }\setminus B_{s/2})$, for some $s>0$ to be fixed later. Notice that if we rescale $f_s(x) = f(s x)$ then the scale invariance of the operator $L_C$ implies that 

\begin{equation*}
  \sup_{K_s} f \leq C_{K_1} \inf_{K_s} f.
 \end{equation*}

We want to evaluate the $L^2$-norm of $f$ on $K_s$ with respect to the cone metric $g_C = dr^2 + r^2 g_\Sigma.$  First note that 
\begin{equation*}
\begin{aligned}
    \displaystyle||f||_{L^2(K_s)} \leq  (\mathcal{H}^n(C \cap K_1) s^n\big)^{1/2} \sup_{K_s} f = C_{(K_1, n , \Sigma)}s^{n/2}\sup_{K_s} f,
 \end{aligned}
\end{equation*}
where $C_{(K_1, n, \Sigma)}$ denotes a constant that depends on $K_1, n, \Sigma$ that may vary from line to line.  
On the other hand, since $\phi_j$ is an orthonormal basis of $L^2(\Sigma)$, we have 
\begin{equation*}
\begin{aligned}
    ||f||_{L^2(K_s)} &= \bigg( \int_{s/2}^{2s} \sum_{j=1}^{\infty} (c_j ^\pm)^2 r^{-2\gamma_j^\pm}r^{n-1}dr \bigg)^{1/2} \\
    &= \bigg(\sum_{j=1}^{\infty} (c_j^\pm)^2 s ^{n-2\gamma_j ^\pm} \bigg ( \dfrac{2^{n-2\gamma_j^\pm} - 2^{2 \gamma_j^\pm - n}}{n-2\gamma_j^\pm} \bigg) \bigg)^{1/2},
\end{aligned}
\end{equation*}
and since $\dfrac{2^x - 2^{-x}}{x} \geq 1$ for any $x \in \mathbb R \setminus \{0\}$ we conclude that 
\[||f||_{L^2(K_s)} \geq s^{n/2} \bigg(\sum_{j=1}^{\infty} (c_j^\pm)^2s^{-2\gamma_j^\pm}\bigg)^{1/2}.\] 
The three inequalities thus give
 \begin{equation*}
 C_{(K_1, n, \Sigma)} \bigg(\sum_{j=1}^\infty (c_j ^\pm)^2 s^{-2\gamma_j^\pm} \bigg)^{1/2} \leq \inf_{K_s} f \leq f(r, \omega), 
 \end{equation*}
for all $r \in [s/2, s]$ and $\omega \in \Sigma.$
Multiplying the latter with $\phi_1$, and integrating over $\Sigma$, we get 

\begin{equation*}
C_{(K_1, n , \Sigma)}\bigg (\sum_{j=1}^\infty(c_j ^\pm)^2 s^{-2\gamma_j^\pm}\bigg)^{1/2} \leq c_1^+  r^{-
\gamma_1^+}+ c_1^- r^{-\gamma_1^-},
\end{equation*}
for all $r \in [\frac{s}{2}, 2s].$ Thus we may take $r = s$ and get that 
\begin{equation}
\begin{aligned}
\label{eq: main ineq}
C_{(K_1, n , \Sigma)}\bigg (\sum_{j=1}^\infty(c_j ^\pm)^2 s^{-2\gamma_j^\pm}\bigg)^{1/2} \leq c_1^+  s^{-
\gamma_1^+}+ c_1^- s^{-\gamma_1^-}.
\end{aligned}
\end{equation}
Multiplying now \eqref{eq: main ineq} by $s^{\gamma_1^+}$ we have that 
\begin{equation*}
    C_{(K_1, n , \Sigma)}\bigg (\sum_{j=1}^\infty(c_j ^\pm)^2 s^{2\gamma_1^+ -2\gamma_j^\pm}\bigg)^{1/2} \leq c_1^+ + c_1^- s^{\gamma_1^+-\gamma_1^-}. 
\end{equation*}

In order to prove that $c_j^+ =0$ for all $j \geq 2$ first note that $\sum_{j=2}^\infty (c_j^+)^2 < \infty$ (by Parseval's identity it is bounded by $||f||_{L^2(\Sigma)}$), and recall that $\gamma_j^-\leq \gamma_2^- < \gamma_1^- < \gamma_1 ^+ < \gamma_2 ^+ \leq  \gamma_j ^+$ for all $j\geq 2$. Thus for any $E >0$ there exists $s_0>0$ such that $s \leq s_0 \Rightarrow s^{2\gamma_1^+ - 2 \gamma_j^+} > E^2$ for every $j$, and moreover $s^{\gamma_1^+ - \gamma_1^-} < \frac{1}{|c_1^-|}$ thus we obtain
\[E^2 \sum_{j=2}^\infty (c_j^+)^2 \leq C_{(K_1, n, \Sigma)} (c_1^+ + 1 ) ^2, \] 
which gives a contradiction for sufficiently large $E$ unless $c_j ^+= 0$ for all $j \geq 2.$ If we instead multiply \eqref{eq: main ineq} by $s^{\gamma_1^-}$ and choose $s$ sufficiently large a similar argument leads to a contradiction unless $c_j^- = 0$ for all $j \geq 2.$ 

It remains to show that $c_1^+, c_1^- \geq 0$. Assume for the contrary that $c_1^+ <0$ then 
\[r^{\gamma^+}f = c_1^+ \phi_1 + c_1^- r ^{\gamma^+ - \gamma^-}\phi_1, \]
and letting $r \to 0$ we get a contradiction. A similar argument gives $c_1^- \geq 0.$ 

In case the cone is not strictly stable, thus $\gamma^+=\gamma^- = \frac{n-2}{2} $, then the expression of the function $f$ is given by 
\[f(r\omega) = c_1^+ r^{-\gamma}\phi_1(\omega) + c_1^- \log r \:\phi_1(\omega) + \sum_{j=2}^\infty c_j^{\pm}r^{-\gamma_j^\pm}\phi_j (\omega) \]
and repeating the same computations as above we will get that $c_1^- = c_j^\pm =0$, for all $j\geq 2$, thus $f(r\omega) = c_1^+ r^{-\gamma}\phi_1(\omega)$, where $c_1^+$ is a non-negative constant. This concludes the proof of Lemma \ref{lem: rep of J-F}.
\end{proof}

\begin{small}

\end{small}

\end{document}